\documentclass[reqno]{amsart}
\usepackage{amssymb}
\usepackage{hyperref}
\usepackage{booktabs}
 \usepackage{multirow}

 \usepackage{color}

\newtheorem{theorem}{Theorem}[section]
\newtheorem{proposition}[theorem]{Proposition}
\newtheorem{lemma}[theorem]{Lemma}

\theoremstyle{remark}
\newtheorem{remark}{Remark}[section]

\numberwithin{equation}{section}

\begin{document}

\title[Exact Cubature Rules for Symmetric  Functions]
{Exact Cubature Rules for Symmetric Functions}

\author{J.F.  van Diejen}

\address{
Instituto de Matem\'atica y F\'{\i}sica, Universidad de Talca,
Casilla 747, Talca, Chile}

\email{diejen@inst-mat.utalca.cl}

\author{E. Emsiz}

\address{
Facultad de Matem\'aticas, Pontificia Universidad Cat\'olica de Chile,
Casilla 306, Correo 22, Santiago, Chile}
\email{eemsiz@mat.uc.cl}


\thanks{This work was supported in part by the {\em Fondo Nacional de Desarrollo
Cient\'{\i}fico y Tecnol\'ogico (FONDECYT)} Grants   \# 1170179 and \# 1181046.}

\date{April 2018}

\begin{abstract}
We employ a multivariate extension of the Gauss quadrature formula, originally due to Berens, Schmid and Xu \cite{ber-sch-xu:multivariate},
so as to derive cubature rules for the integration of symmetric functions over hypercubes (or infinite limiting degenerations thereof)  with respect to
the densities of unitary random matrix ensembles. Our main application concerns the explicit implementation of a class of cubature rules associated with the
Bernstein-Szeg\"o polynomials, which permit the exact integration of  symmetric rational functions
with prescribed poles at coordinate hyperplanes against unitary circular Jacobi distributions
stemming from the Haar measures on the symplectic and the orthogonal groups. 
\end{abstract}

\maketitle



\section{Introduction}\label{sec1}
The study of cubature rules for the numeric integration of functions in several variables has a long fruitful history, see e.g. \cite{str:approximate,dav-rab:methods,sob:cubature,sob-vas:theory,coo:constructing,coo-mys-sch:cubature,ise-nor:quadrature,dun-xu:orthogonal,col-hub:moment} and references therein.
Over the past few years significant progress has been reported regarding the construction of explicit cubature
rules of Gauss-Chebyshev type, permitting the exact integration of multivariate polynomials \cite{li-xu:discrete,moo-pat:cubature,noz-saw:note,moo-mot-pat:gaussian,hri-mot:discrete,hri-mot-pat:cubature}.

Inspired by these recent developments we invoke the Cauchy-Binet-Andr\'eief formulas to rederive a multivariate lifting of the Gauss quadrature formula due to Berens, Schmid and Xu \cite{ber-sch-xu:multivariate}, in the version
designed to integrate symmetric functions over  a hypercube, a hyperoctant, or over the entire euclidean space. In the case of the classical Gauss-Hermite, the Gauss-Laguerre, and the Gauss-Jacobi quadratures \cite{sze:orthogonal,dav-rab:methods}, this readily produces corresponding cubature rules for the exact integration of symmetric polynomials against the densities of ubiquitous  unitary random matrix ensembles associated with the Hermite, Laguerre and Jacobi polynomials  \cite{meh:random,for:log-gases}, respectively. 

At the special parameter values for which the Gauss-Jacobi quadrature simplifies to a Gauss-Chebyshev quadrature, the  construction in question leads to
cubature rules associated with the classical simple Lie groups that turn out to be closely related to those studied in 
 \cite{moo-pat:cubature,moo-mot-pat:gaussian,hri-mot:discrete,hri-mot-pat:cubature}.  
One of our  primary concerns is to extend the corresponding Gauss-Chebyshev cubatures to
a class of explicit cubature rules arising from the Bernstein-Szeg\"o polynomials \cite[Section 2.6]{sze:orthogonal}. It is well-known that the Gauss quadrature associated with
the Bernstein-Szeg\"o polynomials 
\cite{not:error,peh:remainder,not:interpolatory,ber-cac-mar:new,ber-cac-gar-mar:new} permits the exact integration of rational functions with prescribed poles (outside the integration domain)
\cite{dar-gon-jim:quadrature,bul-cru-dec-gon:rational} (cf. also \cite{gau:gauss-type,van-van:quadrature,gau:use,bul-gon-hen-nja:quadrature} for related approaches). Our aim is to extend this picture to the multivariate setup: we construct an
exact cubature rule for a class of symmetric rational functions
with prescribed poles at the coordinate hyperplanes, where the integration is against the unitary circular Jacobi distributions
stemming from the Haar measures on the symplectic and the orthogonal groups (cf. e.g. \cite[Chapter IX.9]{sim:representations}, \cite[Chapter 11.10]{pro:lie} or \cite[Chapter 2.6]{for:log-gases}).
This way the above-mentioned Gauss-Chebyshev cubature rules originating from parameter specializations of the (more involved) Gauss-Jacobi cubature formulas are generalized so as to allow
for pole singularities on the coordinate hyperplanes.

The presentation is structured as follows. After setting up the notation for the Gauss quadrature rule in Section \ref{sec2}, we emphasize  in Section \ref{sec3} the effectiveness of the Cauchy-Binet-Andr\'eief formulas  when extending the underlying family of orthogonal polynomials to the multivariate level via associated generalized Schur polynomials 
\cite{mac:schur,ber-sch-xu:multivariate,nak-nou-shi-yam:tableau,ser-ves:jacobi}.
 This readily allows to recover a Gaussian cubature rule 
 for the integrations of symmetric functions from \cite[Equation (8)]{ber-sch-xu:multivariate} (with $\rho=0$) in Section \ref{sec4}.
 In Section \ref{sec5}
we highlight the explicit cubature rules stemming  the classical Hermite, Laguerre and Jacobi families, which permit 
the exact  integration of symmetric polynomials with respect to the densities of the corresponding unitary ensembles.
In the remainder of the paper the implementation of the construction for the case of Bernstein-Szeg\"o polynomials is carried out. Specifically, after recalling the definition of the
Bernstein-Szeg\"o polynomials in Section \ref{sec6} and providing estimates for the locations of their roots in Section \ref{sec7}, the corresponding Gauss quadrature rule
stemming from Refs. \cite{dar-gon-jim:quadrature,bul-cru-dec-gon:rational} is exhibited in Section \ref{sec8}.  In Section \ref{sec9} we then apply the general formalism of Section \ref{sec4}
to lift this quadrature to an explicit cubature rule for an associated class of symmetric rational functions. To enhance the readability, some technical details regarding the explicit computation of  the pertinent Christoffel weights associated with  the Bernstein-Szeg\"o families are supplemented in Appendix \ref{appA} at the end.

\section{Preliminaries and notation regarding the Gauss quadrature}\label{sec2}
Given 
a continuous weight function $\text{w}(x)>0$  on a nonempty interval $(a,b)$ with finite moments, let
$p_l(x)$, $l=0,1,2,\ldots$ denote the orthonormal basis obtained from the monomial basis $m_l(x):=x^l$, $l=0,1,2,\ldots$ 
via Gram-Schmidt orthogonalization with respect to the inner product
\begin{equation}
(f,g)_{\text{w}}:= \int_a^b f(x) g(x) \text{w}(x)\text{d} x
\end{equation}
(for $f,g:(a,b)\to\mathbb{R}$ polynomial (say)). 
It is well-known (cf. e.g. \cite[Section 3.3]{sze:orthogonal}) that  the roots of such orthogonal polynomials $p_l(x)$ are simple and belong to $(a,b)$, i.e.
for  $m\geq 0$:
\begin{subequations}
\begin{equation}\label{factorization}
p_{m+1}(x)= \alpha_{m+1} \bigl(x-x_0^{(m+1)}\bigr) \bigl(x-x_1^{(m+1)}\bigr)\cdots  \bigl(x-x_m^{(m+1)}\bigr),
\end{equation}
with
\begin{equation}\label{s-roots}
a< x_0^{(m+1)}< x_1^{(m+1)}<\cdots <  x_{m}^{(m+1)}< b
\end{equation}
\end{subequations}
and $\alpha_l:=1/(p_l,m_l)_{\text{w}}$ ($l=0,1,2,\ldots$).

Let $f(x)$ be an arbitrary polynomial  of degree at most $2m+1$ in $x$. 
The celebrated Gauss quadrature formula states that in this situation (cf. e.g. \cite[Section 3.4]{sze:orthogonal}, \cite{gau:survey}, or
for an overview of more recent developments \cite{gau:orthogonal}):
\begin{subequations}
\begin{equation}\label{gqr:a}
 \int_a^b f(x)\text{w}(x) \text{d}x =  \sum_{0\leq \hat{l}\leq m}    f(x_{\hat{l}}^{(m+1)})   {\text{w}}^{(m+1)}_{\hat{l}},
\end{equation}
where the corresponding Christoffel weights $ {\text{w}}^{(m+1)}_0,\ldots , {\text{w}}^{(m+1)}_m$ are given by
\begin{equation}\label{gqr:b}
{\text{w}}^{(m+1)}_{\hat{l}} =  \Big( \sum_{0\leq l\leq m} \bigl( p_l(x_{\hat{l}}^{(m+1)})\bigr)^2\Bigr)^{-1}\quad (\hat{l}=0,\ldots ,m).
\end{equation}
\end{subequations}
This quadrature rule can be reformulated in terms of discrete orthogonality relations for $p_0(x),p_1(x),\ldots ,p_m(x)$.
Indeed, when applying the Gauss quadrature rule \eqref{gqr:a}, \eqref{gqr:b} to the product $f(x)=p_l(x)p_{k}(x)$ with $l,k$ at most $m$, the defining orthogonality
\begin{equation}
(p_l,p_k)_{\text{w}}= \begin{cases}
1&\text{if}\ k =l, \\
0&\text{if}\ k \neq l
\end{cases}
\end{equation}
 gives rise to the associated finite-dimensional discrete orthogonality:
\begin{subequations}
\begin{equation}\label{finite-orthogonality}
  \sum_{0\leq \hat{l}\leq m}    p_l(x_{\hat{l}}^{(m+1)}) p_{k} (x_{\hat{l}}^{(m+1)})   {\text{w}}^{(m+1)}_{\hat{l}}=
\begin{cases}
1&\text{if}\ k =l, \\
0&\text{if}\ k \neq l
\end{cases}
\end{equation}
($l,k \in \{ 0,\ldots ,m\}$). 
 By `column-row duality',  one can further reformulate Eq. \eqref{finite-orthogonality} in terms of the equivalent dual orthogonality relations
\begin{equation}\label{finite-orthogonality:dual}
  \sum_{0\leq l\leq m}    p_l(x_{\hat{l}}^{(m+1)}) p_{l} (x_{\hat{k}}^{(m+1)})  =
\begin{cases}
 1/ {\text{w}}^{(m+1)}_{\hat{l}} &\text{if}\ \hat{k} =\hat{l}, \\
0&\text{if}\ \hat{k} \neq \hat{l}
\end{cases}
\end{equation}
\end{subequations}
($\hat{l},\hat{k} \in \{ 0,\ldots ,m\}$).

\begin{remark}\label{m+1:rem}
The discrete orthogonality relations in Eq. \eqref{finite-orthogonality} remain in fact valid when either $l$ or $k$ (but not both) become equal to $m+1$ (because $p_{m+1}(x)$ vanishes identically on the nodes $x_0^{(m+1)},\ldots ,x_m^{(m+1)}$).
\end{remark}

\section{Generalized Schur polynomials}\label{sec3}
Given  $ \lambda =(\lambda_1,\ldots,\lambda_n)$ in the fundamental cone
\begin{equation}\label{f-cone}
\Lambda^{(n)}:=\{ \lambda  \in\mathbb{Z}^n \mid \lambda_1\geq \lambda_2\geq\cdots \geq \lambda_n\geq 0 \} ,
\end{equation}
the generalized Schur polynomial $P_\lambda(\mathbf{x})$ associated with the orthonormal system $p_0(x),p_1(x),p_2(x),\ldots $
is defined via the determinantal formula (cf. \cite{mac:schur,ber-sch-xu:multivariate,nak-nou-shi-yam:tableau,ser-ves:jacobi})
\begin{subequations}
\begin{equation}\label{schur}
P_\lambda(\mathbf{x}) := \frac{1}{\text{V}(\mathbf{x})} \det [ p_{\lambda_j+n-j}(x_k)]_{1\leq j,k\leq n} ,
\end{equation}
where $V(\mathbf{x})$ refers to the Vandermonde determinant
\begin{equation}\label{vandermonde}
\text{V}(\mathbf{x}):=\prod_{1\leq j <k \leq n}  (x_j-x_k) .
\end{equation}
\end{subequations}
Clearly $P_\lambda (\mathbf{x})$ constitutes a (permutation-)symmetric polynomial in the components of $\mathbf{x}:=(x_1,x_2,\ldots ,x_n)$ (as the determinant in the numerator produces an antisymmetric polynomial which is therefore divisible by the Vandermonde determinant).

It is well-known that the symmetric polynomials in question inherit multivariate orthogonality relations from the underlying univariate family
(cf. e.g. \cite{ber-sch-xu:multivariate,ser-ves:jacobi}). 

\begin{proposition}[Orthogonality Relations]\label{GS-orthogonality:prp}
The generalized Schur polynomials  $P_\lambda(\mathbf{x})$, $\lambda\in \Lambda^{(n)}$ satisfy the  orthogonality relations
\begin{subequations}
\begin{align}
\frac{1}{n!} \int_a^{b}\cdots\int_a^b 
P_\lambda (\mathbf{x})  P_\mu (\mathbf{x})  \emph{W}(\mathbf{x}) \text{d} x_1\cdots\text{d} x_n 
 =\begin{cases}   1 &\text{if}\ \mu = \lambda\\  0  &\text{if}\  \mu\neq \lambda\end{cases}  
\end{align}
($\lambda,\mu \in \Lambda^{(n)}$), where
\begin{equation}\label{W}
\emph{W}(\mathbf{x}): =\bigl(  \emph{V}(\mathbf{x}) \bigr)^2
 \prod_{1\leq j\leq n}\emph{w}(x_j) .
\end{equation}
\end{subequations}
\end{proposition}
\begin{proof}
This orthogonality is immediate from $(i)$ a classical (Cauchy-Binet type) integration formula for the products of determinants going back to M.C.  Andr\'eief
\cite{and:note} (which is for instance reproduced with proof in \cite[Lemma 3.1]{bai-dei-stra:products}), in combination with $(ii)$  the orthogonality of the basis polynomials  $p_0(x),p_1(x),p_2(x),\ldots$. More specifically, one has that $\forall \lambda,\mu\in\Lambda^{(n)}$:
\begin{align*}
&\frac{1}{n!} \int_a^{b}\cdots\int_a^b 
P_\lambda (\mathbf{x})  P_\mu (\mathbf{x})  \text{W}(\mathbf{x}) \text{d} x_1\cdots\text{d} x_n \\
&= \frac{1}{n!} \int_a^{b}\cdots\int_a^b 
 \det [ p_{\lambda_j+n-j}(x_k)]_{1\leq j,k\leq n}   \det [ p_{\mu_j+n-j}(x_k)]_{1\leq j,k\leq n}  \prod_{1\leq j\leq n}\text{w}(x_j) \text{d} x_j\\
 &\stackrel{(i)}{=} \det \left[  \int_a^b p_{\lambda_j+n-j}(x)  p_{\mu_k+n-k}(x)  \text{w}(x)\text{d} x   \right]_{1\leq j,k\leq n}
 \stackrel{(ii)}{=} \begin{cases}   1  &\text{if}\ \mu = \lambda , \\  0  &\text{if}\  \mu\neq \lambda .\end{cases}  
\end{align*}
\end{proof}

The following proposition provides a corresponding multivariate generalization of the finite-dimensional discrete orthogonality in Eq. \eqref{finite-orthogonality}, which holds for $P_\lambda (\mathbf{x})$ when $\lambda$ is restricted to the fundamental alcove
\begin{equation}\label{f-alcove}
\Lambda^{(m,n)}:=\{ \lambda =(\lambda_1,\ldots,\lambda_n) \in\mathbb{Z}^n \mid m \geq \lambda_1\geq \lambda_2\geq\cdots \geq \lambda_n\geq 0 \} .
\end{equation}

\begin{proposition}[Discrete Orthogonality Relations]\label{GS-discrete-orthogonality:prp}
The generalized Schur polynomials  $P_\lambda(\mathbf{x})$, $\lambda\in \Lambda^{(m,n)}$
satisfy the  discrete orthogonality relations
\begin{subequations}
\begin{align}
\sum_{\hat{\lambda}\in \Lambda^{(m,n)}}
P_\lambda \bigl(\mathbf{x}^{(m,n)}_{\hat{\lambda}}\bigr)  P_\mu \bigl(\mathbf{x}^{(m,n)}_{\hat{\lambda}}\bigr)  \emph{W}_{\hat{\lambda}}
 =\begin{cases}   1 &\text{if}\ \mu = \lambda\\  0  &\text{if}\  \mu\neq \lambda\end{cases}  
\end{align}
($\lambda,\mu \in \Lambda^{(m,n)}$), where 
\begin{equation}\label{X-nodes}
\mathbf{x}^{(m,n)}_{\hat{\lambda}}:= \left(x^{(m+n)}_{\hat{\lambda}_1+n-1},x^{(m+n)}_{\hat{\lambda}_2+n-2},\ldots ,x^{(m+n)}_{\hat{\lambda}_{n-1}+1},x^{(m+n)}_{\hat{\lambda}_n}\right)
\end{equation}
and
\begin{equation}\label{W-discrete}
\emph{W}_{\hat{\lambda}}:= \Bigl( \emph{V}\bigl( \mathbf{x}^{(m,n)}_{\hat{\lambda}} \bigr) \Bigr)^2  \prod_{1\leq j\leq n}  \emph{w}^{(m+n)}_{\hat{\lambda}_j+n-j}
\end{equation}
\end{subequations}
(for $\hat{\lambda}\in \Lambda^{(m,n)}$). Here $x^{(m+n)}_{\hat{\lambda}_j+n-j}$ and $\emph{w}^{(m+n)}_{\hat{\lambda}_j+n-j}$ ($j=1,\ldots ,n$) are in accordance with the definitions in Eqs. \eqref{factorization}, \eqref{s-roots} and Eqs. \eqref{gqr:a}, \eqref{gqr:b}, respectively.
\end{proposition}
\begin{proof}
Similarly to the proof of Proposition \ref{GS-orthogonality:prp},
the asserted orthogonality relations are derived  from  those
in Eq. \eqref{finite-orthogonality} by means of the Cauchy-Binet formula:
\begin{equation*}
\sum_{\hat{\lambda}\in \Lambda^{(m,n)}}
P_\lambda \bigl(\mathbf{x}^{(m,n)}_{\hat{\lambda}}\bigr)  P_\mu \bigl(\mathbf{x}^{(m,n)}_{\hat{\lambda}}\bigr)  \text{W}_{\hat{\lambda}}
 =
 \end{equation*}
 \begin{align*}
 \sum_{m+n>\tilde{\lambda}_1>\tilde{\lambda}_2>\cdots >\tilde{\lambda}_n\geq 0} &\left(
 \det \Biggl[ p_{\lambda_j+n-j}\Bigl(x^{(m+n)}_{\tilde{\lambda}_k}\Bigr)  \sqrt{\text{w}^{(m+n)}_{\tilde{\lambda}_k}}\Biggr]_{1\leq j,k\leq n}  \right. \\
& \left. \times \det \Biggl[ p_{\mu_j+n-j}\Bigl( x^{(m+n)}_{\tilde{\lambda}_k}\Bigr) \sqrt{\text{w}^{(m+n)}_{\tilde{\lambda}_k}}\Biggr]_{1\leq j,k\leq n}   \right) 
 \end{align*}
 \begin{align*}
& \stackrel{(i)}{=}\det  \left[     \sum_{0\leq \hat{l} < m+n}  p_{\lambda_j+n-j}\Bigl(x^{(m+n)}_{\hat{l}} \Bigr)  p_{\mu_k+n-k}\Bigl( x^{(m+n)}_{\hat{l}}  \Bigr)   \text{w}^{(m+n)}_{\hat{l}}  \right]_{1\leq j,k\leq n}  \\
&\stackrel{(ii)}{=} \begin{cases}   1 &\text{if}\ \mu = \lambda  \\  0  &\text{if}\  \mu\neq \lambda \end{cases}  
 \end{align*}
 (where it was assumed that $\lambda,\mu\in\Lambda^{(m,n)}$).
Here the equality $(i)$  hinges on the Cauchy-Binet formula while the equality $(ii)$ follows using Eq. \eqref{finite-orthogonality}.
\end{proof}

\begin{remark}\label{GS-dual-discrete-orthogonality:rem}
The alternative dual formulation of the discrete orthogonality in Proposition \ref{GS-discrete-orthogonality:prp} reads (cf. Eq. \eqref{finite-orthogonality:dual}):
\begin{align}
\sum_{\lambda\in \Lambda^{(m,n)}}
P_\lambda \bigl(\mathbf{x}^{(m,n)}_{\hat{\lambda}}\bigr)  P_\lambda \bigl(\mathbf{x}^{(m,n)}_{\hat{\mu}}\bigr) 
 =\begin{cases}   1/ \text{W}_{\hat{\lambda}} &\text{if}\ \hat{\mu} =\hat{ \lambda}\\  0  &\text{if}\  \hat{\mu}\neq \hat{\lambda}\end{cases}  
\end{align}
($\hat{\lambda},\hat{\mu} \in \Lambda^{(m,n)}$).
\end{remark}

\begin{remark}\label{GS-m+1:rem}
The orthogonality in Proposition \ref{GS-discrete-orthogonality:prp} extends in fact to the situation that $\lambda\in\Lambda^{(m+1,n)}$ and $\mu\in \Lambda^{(m,n)}$. Indeed, it is immediate from the definitions that
 if $\lambda_1=m+1$ then 
$P_\lambda \bigl(\mathbf{x}^{(m,n)}_{\hat{\lambda}}\bigr) =0$ for all $\hat{\lambda}\in \Lambda^{(m,n)}$ (cf. Remark \ref{m+1:rem}).
\end{remark}

\section{Gaussian cubature for symmetric functions}\label{sec4}
For $\lambda\in \Lambda^{(n)}$, let us define the symmetric monomial
\begin{equation}\label{monomials}
M_\lambda (\mathbf{x}):=  \sum_{\mu \in S_n \lambda} x_1^{\mu_1} x_2^{\mu_2}\cdots x_n^{\mu_n},
\end{equation}
where the summation is meant over  the orbit  $S_n\lambda$ of $\lambda$ with respect to the standard action of the permutation group $S_n$ on the components:
\begin{equation}
\lambda=(\lambda_1,\ldots ,\lambda_n) \stackrel{\sigma}{\to} (\lambda_{\sigma^{-1}(1)},\ldots ,\lambda_{\sigma^{-1}(n)})=:\sigma \lambda 
\end{equation}
for $\sigma = \left( \begin{matrix} 1& 2& \cdots & n \\
 \sigma (1)&\sigma (2)&\cdots & \sigma (n)
 \end{matrix}\right) \in S_n$. Clearly  $M_\lambda (\mathbf{x})$ is homogeneous of total degree
 \begin{equation}
 |\lambda|:=\lambda_1+\lambda_2+\cdots +\lambda_n.
 \end{equation}
Upon restricting the (inhomogeneous) dominance order
 \begin{equation}
\forall\mu,\lambda\in\mathbb{Z}^n:\qquad \mu\leq\lambda \Leftrightarrow \sum_{1\leq j\leq k} (\lambda_j-\mu_j)\geq 0\quad\text{for}\quad k=1,\ldots ,n
 \end{equation}
from $\mathbb{Z}^n$ to $\Lambda^{(n)}$,  this partial ordering is inherited by monomial basis $M_\lambda (\mathbf{x})$, $\lambda\in\Lambda^{(n)}$.

 Let $\mathbb{P}^{(m,n)}$ denote the $\binom{m+n}{n}$-dimensional subspace of the algebra of symmetric polynomials spanned by $M_\lambda (\mathbf{x})$, $\lambda\in\Lambda^{(m,n)}$.  Notice that these are precisely the monomials $M_\lambda (\mathbf{x})$ with $\lambda\in\Lambda^{(n)}$
 such that  $\lambda\subseteq (m)_n:=(m,\ldots,m)\in\Lambda^{(n)}$ (where for $\lambda,\mu\in\Lambda^{(n)}$ one writes $\lambda\subseteq\mu$ iff $\lambda_j\leq\mu_j$ for $j=1,\ldots,n$). In other words, $\mathbb{P}^{(m,n)}$ consists of all symmetric polynomials of degree at most $m$ in each of the variables $x_j$ ($j\in \{ 1,\ldots ,n\}$).
 
 \begin{lemma}[Generalized Schur basis]\label{GS-basis:lem}
 The generalized Schur polynomials $P_\lambda (\mathbf{x})$, $\lambda\in\Lambda^{(m,n)}$ constitute a basis for $\mathbb{P}^{(m,n)}$.
 \end{lemma}
 
 \begin{proof}
Since the numerator is a polynomial in $x_j$ of degree at most $\lambda_1+n-1$ and the Vandermonde determinant is of degree $n-1$ in $x_j$,  it is clear that the generalized Schur polynomial  $P_\lambda (\mathbf{x})$, \eqref{schur}, \eqref{vandermonde} belongs to $\mathbb{P}^{(m,n)}$ when $\lambda\in\Lambda^{(m,n)}$.
Moreover, 
if we replace $p_l(x)$ on the RHS of Eq. \eqref{schur} by $x^l$, then we recover a classic determinantal formula for the conventional Schur polynomial $S_\lambda (\mathbf{x})$
(cf. e.g. \cite[Equation (0.1)]{mac:schur}). Hence---up to normalization---the top-degree terms of $P_\lambda (\mathbf{x})$ are given by $S_\lambda (\mathbf{x})$.
It is therefore enough to infer that the Schur polynomials $S_\lambda (\mathbf{x})$, $\lambda\in\Lambda^{(m,n)}$ provide a basis for $\mathbb{P}^{(m,n)}$. This, however, is immediate  from the well-known
fact that the expansion of the Schur polynomials on the monomial basis is  unitriangular with respect to the dominance partial order (cf. e.g. \cite[Chapter I.6]{mac:symmetric}):
\begin{equation*}
S_\lambda (\mathbf{x})= M_\lambda (\mathbf{x})+ \sum_{\substack{\mu\in\Lambda^{(n)}, \, \mu <\lambda\\ |\mu|=|\lambda|}} C^\mu_\lambda M_\mu (\mathbf{x})
\end{equation*}
for certain (nonnegative integral) coefficients $C^{\mu}_{\lambda} $.
 \end{proof}

After these preparations we are now in the position to reformulate the orthogonality relations of Propositions \ref{GS-orthogonality:prp} and \ref{GS-discrete-orthogonality:prp} 
as a cubature rule for the integration of symmetric functions in $n$  variables over
$(a,b)^n\subseteq\mathbb{R}^n$. The resulting cubature formula---which provides a multivariate extension of the celebrated Gauss quadrature rule \eqref{gqr:a}, \eqref{gqr:b}---was originally found
by Berens, Schmid and Xu,  cf. \cite[Equation (8)]{ber-sch-xu:multivariate} (with $\rho=0$) and \cite[Chapter 5.4]{dun-xu:orthogonal}. 

\begin{proposition}[Exact  Gaussian Cubature Rule in $\mathbb{P}^{(2m+1,n)}$]\label{gauss-cubature:prp}
For $f(\mathbf{x})\in\mathbb{P}^{(2m+1,n)}$, one has that
\begin{equation}\label{GS-cubature}
\frac{1}{n!} \int_a^{b}\cdots\int_a^b 
f (\mathbf{x})  \emph{W}(\mathbf{x}) \emph{d} x_1\cdots\emph{d} x_n 
 =\sum_{\hat{\lambda}\in \Lambda^{(m,n)}}
f \bigl(\mathbf{x}^{(m,n)}_{\hat{\lambda}}\bigr)   \emph{W}_{\hat{\lambda}} ,
\end{equation}
where $\emph{W}(\mathbf{x})$, $\mathbf{x}^{(m,n)}_{\hat{\lambda}}$ and $\emph{W}_{\hat{\lambda}} $ are drawn from Eqs. \eqref{W}, \eqref{X-nodes}, and \eqref{W-discrete}, respectively.
\end{proposition}

\begin{proof}
By comparing the orthogonality relations in Propositions 
\ref{GS-orthogonality:prp} and \ref{GS-discrete-orthogonality:prp}---while also recalling Remark \ref{GS-m+1:rem}---it is plain that the cubature rule in Eq. \eqref{GS-cubature} is valid for $f(\mathbf{x})=P_\lambda (\mathbf{x})P_\mu(\mathbf{x})$
with $\lambda\in\Lambda^{(m+1,n)}$ and $\mu\in\Lambda^{(m,n)}$. By Lemma \ref{GS-basis:lem} and the bilinearity, the same is thus true for $f(\mathbf{x})=M_\lambda (\mathbf{x})M_\mu(\mathbf{x})$
with $\lambda\in\Lambda^{(m+1,n)}$ and $\mu\in\Lambda^{(m,n)}$. Since $\sigma\lambda \leq \lambda$ for all $\lambda\in\Lambda^{(n)}$ and $\sigma\in S_n$
(cf. e.g. \cite[Chapter III.13.2]{hum:introduction}), it is clear from Eq. \eqref{monomials} that
\begin{equation*}
M_\lambda (\mathbf{x})M_\mu(\mathbf{x})=M_{\lambda+\mu} (\mathbf{x})+\sum_{\substack{\nu\in\Lambda^{(n)} ,\, \nu< \lambda+\mu\\ |\nu|=|\lambda +\mu|}}  C^{\nu}_{\lambda,\mu} M_{\nu} (\mathbf{x})
\end{equation*}
for certain (nonnegative integral) coefficients $C^{\nu}_{\lambda,\mu} $. Hence, 
the products  $M_\lambda (\mathbf{x})M_\mu(\mathbf{x})$
 span the space $\mathbb{P}^{(2m+1,n)}$ as $\lambda$ and $\mu$ vary over  $\Lambda^{(m+1,n)}$ and $\Lambda^{(m,n)}$, respectively. The asserted cubature rule now follows for general  $f(\mathbf{x})\in\mathbb{P}^{(2m+1,n)}$ by linearity.
 \end{proof}

\begin{remark}\label{coordinate-trafo:rem}
Since any symmetric polynomial can be uniquely written as a polynomial expression in the elementary symmetric monomials,
the change of variables $\mathbf{x}\to\mathbf{y}=(y_1,\ldots, y_n)$ given by
\begin{equation}\label{cov}
y_k=E_k(x_1,\ldots ,x_n):= \sum_{1\leq j_1 < j_2 <\cdots <j_k\leq n}  x_{j_1} x_{j_2}\cdots x_{j_k}\quad   (k=1,\ldots ,n),
\end{equation}
induces a linear isomorphism between  $\mathbb{P}^{(m,n)}$ and the space $\Pi^{(m,n)}$ of all (not necessarily symmetric) polynomials in the variables $y_1,\ldots ,y_n$ of total degree $\leq m$. In particular, $\dim (\Pi^{(m,n)})  = \dim (\mathbb{P}^{(m,n)})= \binom{m+n}{m}$.
Under this change of variables, the cubature formula in Eq. \eqref{GS-cubature} transforms into an exact cubature formula in $\Pi^{(2m+1,n)}$
supported on $\dim (\Pi^{(m,n)})$ nodes that was detailed explicitly in \cite[Equation (2)]{ber-sch-xu:multivariate}. Since it is well-known that any exact cubature rule in $\Pi^{(2m+1,n)}$ involves function evaluations on at least
$\dim (\Pi^{(m,n)})$ nodes (cf. e.g. \cite[Chapter 3.8]{dun-xu:orthogonal} and references therein), it follows via the change of variables in Eq. \eqref{cov} that similarly any exact cubature rule in
 $\mathbb{P}^{(2m+1,n)}$ involves function evaluations on at least
$\dim (\mathbb{P}^{(m,n)})$ nodes.  Following standard terminology  \cite[Chapter 3.8]{dun-xu:orthogonal}, here we refer to exact cubature rules in $\mathbb{P}^{(2m+1,n)}$ supported on  precisely (the minimal possible number of) $\dim (\mathbb{P}^{(m,n)})$ nodes as being \emph{Gaussian}.  From this perspective, Propostion \ref{gauss-cubature:prp} is to be viewed as a concrete example of a Gaussian cubature rule in $\mathbb{P}^{(2m+1,n)}$. Notice in this connection also that---in view of  Remark \ref{GS-m+1:rem}---the nodes $\mathbf{x}^{(m,n)}_{\hat{\lambda}}$, $\hat{\lambda}\in\Lambda^{(m,n)}$ are common zeros of all  $\binom{m+n}{m+1}$ basis polynomials $P_\lambda (\mathbf{x})$, $\lambda\in\Lambda^{(n)}$ that are precisely of degree $m+1$ in each of the variables $x_j$ ($j\in \{ 1,\ldots ,n\}$), cf.
 \cite[Theorem 3.8.4]{dun-xu:orthogonal}.
\end{remark}

\section{The classical orthogonal families: cubature rules for unitary random matrix ensembles}\label{sec5}
By specializing $p_l(x)$, $l=0,1,2,\ldots$ to the classical orthogonal families of Hermite, Laguerre and Jacobi type, Proposition \ref{gauss-cubature:prp} provides cubature rules
for the exact integration of $f(\mathbf{x})\in\mathbb{P}^{(2m+1,n)}$ with respect to the densities of the Gaussian unitary ensemble, the Laguerre unitary ensemble,
and the Jacobi unitary ensemble, respectively.

\subsection{Gaussian unitary ensemble}
The normalized Hermite polynomials
\begin{equation*}
h_l(x)={\textstyle \frac{1}{\sqrt{2^l l! }} } H_l(x), \quad l=0,1,2,\ldots
\end{equation*}
 constitute an orthonormal basis on the interval $(a,b)=(-\infty,\infty)$ with respect to the weight function
 $w(x)=\frac{1}{\sqrt{\pi}} e^{-x^2}$  \cite[Chapter 18]{olv-loz-boi-cla:nist}.
 At the $(\hat{l}+1)$th root $x^{(m+1)}_{\hat{l}}$ of $h_{m+1}(x)$ the corresponding Christoffel weight is given by
 (cf. e.g. \cite[Chapter 15.3]{sze:orthogonal} or \cite[Chapter 3.6]{dav-rab:methods})
 \begin{equation*}
 \text{w}^{(m+1)}_{\hat{l}}=\Bigl((m+1)h^2_m\bigl(  x^{(m+1)}_{\hat{l}}  \bigr)\Bigr)^{-1}\quad (0\leq\hat{l}\leq m).
 \end{equation*}
 In this situation Proposition \ref{gauss-cubature:prp} gives rise to the following Gauss-Hermite cubature rule for the integration of
  $f(\mathbf{x})\in\mathbb{P}^{(2m+1,n)}$ with respect to the density of the Gaussian unitary ensemble
  (cf. e.g. \cite[Chapter 3.3]{meh:random} or \cite[Chapter 1.3]{for:log-gases}): 
  \begin{subequations}
 \begin{align}\label{gauss-hermite-cubature-rule}
 \frac{1}{\pi^{\frac{n}{2}} n!}  & \int_{-\infty}^\infty\cdots \int_{-\infty}^\infty
 f(\mathbf{x})  \prod_{1\leq j\leq n} e^{-x_j^2}    \prod_{1\leq j<k\leq n} (x_j-x_k)^2 \text{d}x_1\cdots\text{d}x_n\\
&=\sum_{\hat{\lambda}\in\Lambda^{(m,n)}}   
f\bigl(\mathbf{x}^{(m,n)}_{\hat{\lambda}}\bigr)  \text{W}_{\hat{\lambda}} , \nonumber
 \end{align}
where
\begin{align}\label{gaus-hermite-cubature-weights}
\text{W}_{\hat{\lambda}}=&
\frac{1}{(m+n)^n} \prod_{1\leq j\leq n} 
 \Bigl( h_{m+n-1} \bigl(  x^{(m+n)}_{\hat{\lambda}_j+n-j}  \bigr)\Bigr)^{-2} \\
 &\times \prod_{1\leq j<k\leq n}  \Bigl(   x^{(m+n)}_{\hat{\lambda}_j+n-j} -x^{(m+n)}_{\hat{\lambda}_k+n-k}       \Bigr)^2. \nonumber
\end{align}
 \end{subequations}

\subsection{Laguerre unitary ensemble}
For $\alpha>-1$ the normalized Laguerre polynomials
\begin{equation*}
\ell^{(\alpha)}_l(x)={\textstyle \sqrt{\frac{l!}{\Gamma (l+1+\alpha)}} }  L^{(\alpha)}_l(x),\quad l=0,1,2,\ldots
\end{equation*}
are orthonormal on the interval $(a,b)=(0,\infty)$
with respect to the weight function
 $w(x)=x^\alpha e^{-x}$  \cite[Chapter 18]{olv-loz-boi-cla:nist}.
 The Christoffel weight at the $(\hat{l}+1)$th root $x^{(m+1)}_{\hat{l}}$ of $\ell^{(\alpha)}_{m+1}(x)$ reads 
 (cf. e.g. \cite[Chapter 15.3]{sze:orthogonal} or \cite[Chapter 3.6]{dav-rab:methods})
 \begin{equation*}
 \text{w}^{(m+1)}_{\hat{l}}= \left((m+1)\, x^{(m+1)}_{\hat{l}}  \left(  \ell^{(\alpha+1)}_m\bigl(  x^{(m+1)}_{\hat{l}}  \bigr)\right)^2 \right)^{-1}\quad (0\leq\hat{l}\leq m).
 \end{equation*}
 The corresponding Gauss-Laguerre cubature rule from Proposition \ref{gauss-cubature:prp} permits
 the exact integration of
  $f(\mathbf{x})\in\mathbb{P}^{(2m+1,n)}$ with respect to the density of the Laguerre unitary ensemble
  (cf. e.g. \cite[Chapter 19]{meh:random} or \cite[Chapter 3]{for:log-gases}): 
  \begin{subequations}
 \begin{align}\label{gauss-laguerre-cubature-rule}
 \frac{1}{ n!} \int_{0}^\infty\cdots \int_{0}^\infty&
 f(\mathbf{x}) \prod_{1\leq j\leq n} x_j^\alpha e^{-x_j}    \prod_{1\leq j<k\leq n} (x_j-x_k)^2 \text{d}x_1\cdots\text{d}x_n\\
&=\sum_{\hat{\lambda}\in\Lambda^{(m,n)}}   
f\bigl(\mathbf{x}^{(m,n)}_{\hat{\lambda}}\bigr)  \text{W}_{\hat{\lambda}} ,  \nonumber
 \end{align}
where
\begin{align}\label{gauss-laguerre-cubature-weights}
\text{W}_{\hat{\lambda}}=&
\frac{1}{(m+n)^n } \prod_{1\leq j\leq n} \left(x^{(m+n)}_{\hat{\lambda}_j+n-j} \right)^{-1}
 \Bigl( \ell^{(\alpha+1)}_{m+n-1} \bigl(  x^{(m+n)}_{\hat{\lambda}_j+n-j}  \bigr)\Bigr)^{-2} \\
&\times  \prod_{1\leq j<k\leq n}  \Bigl(   x^{(m+n)}_{\hat{\lambda}_j+n-j} -x^{(m+n)}_{\hat{\lambda}_k+n-k}       \Bigr)^2. \nonumber
\end{align}
\end{subequations}

\subsection{Jacobi  unitary ensemble}
For $\alpha,\beta >-1$ the normalized Jacobi polynomials
\begin{equation*}
p^{(\alpha,\beta)}_l(x)={\textstyle \sqrt{\frac{(2l+1+\alpha+\beta)\Gamma (l+1+\alpha+\beta)l!}{2^{\alpha+\beta+1}\Gamma (l+1+\alpha)\Gamma (l+1+\beta)}} }
  P^{(\alpha,\beta)}_l(x), \quad
l=0,1,2,\ldots
\end{equation*}
are orthonormal on the interval $(a,b)=(-1,1)$
with respect to the weight function
 $w(x)=(1-x)^\alpha (1+x)^\beta $  \cite[Chapter 18]{olv-loz-boi-cla:nist}.
 The Christoffel weight  at the $(\hat{l}+1)$th root $x^{(m+1)}_{\hat{l}}$ of $p^{(\alpha,\beta)}_{m+1}(x)$ is given by
 (cf. e.g. \cite[Chapter 15.3]{sze:orthogonal})
 \begin{equation*}
 \text{w}^{(m+1)}_{\hat{l}}=\frac{(2m+3+\alpha+\beta)}{(m+1)(m+2+\alpha+\beta)  \left( 1-\bigl( x^{(m+1)}_{\hat{l}}  \bigr)^2\right) \left( p^{(\alpha+1,\beta+1)}_m\bigl(  x^{(m+1)}_{\hat{l}}  \bigr)\right)^2 }
 \end{equation*}
 ($0\leq\hat{l}\leq m$).
 The corresponding Gauss-Jacobi cubature rule from Proposition \ref{gauss-cubature:prp} permits the exact integration of
  $f(\mathbf{x})\in\mathbb{P}^{(2m+1,n)}$ with respect to the density of the Jacobi unitary ensemble
  (cf. e.g. \cite[Chapter 19]{meh:random} or \cite[Chapter 3]{for:log-gases}): 
  \begin{subequations}
 \begin{align}\label{gauss-jacobi-cubature-rule}
 \frac{1}{ n!} \int_{-1}^1 \cdots \int_{-1}^1 &
 f(\mathbf{x}) \prod_{1\leq j\leq n} (1-x_j)^\alpha (1+x_j)^\beta   \prod_{1\leq j<k\leq n} (x_j-x_k)^2 \text{d}x_1\cdots\text{d}x_n\\
&=\sum_{\hat{\lambda}\in\Lambda^{(m,n)}}   
f\bigl(\mathbf{x}^{(m,n)}_{\hat{\lambda}}\bigr)  \text{W}_{\hat{\lambda}} , \nonumber
 \end{align}
where
\begin{align}
\text{W}_{\hat{\lambda}}=&
{\textstyle \frac{(2m+2n+1+\alpha+\beta)^n}{(m+n)^n (m+n+1+\alpha+\beta)^n} }\prod_{1\leq j\leq n}  \left( 1-\bigl( x^{(m+n)}_{\hat{\lambda}_j+n-j}  \bigr)^2\right)^{-1}
 \Bigl( p^{(\alpha+1,\beta+1)}_{m+n-1} \bigl(  x^{(m+n)}_{\hat{\lambda}_j+n-j}  \bigr)\Bigr)^{-2} \nonumber \\
&\times  \prod_{1\leq j<k\leq n}  \Bigl(   x^{(m+n)}_{\hat{\lambda}_j+n-j} -x^{(m+n)}_{\hat{\lambda}_k+n-k}       \Bigr)^2. \label{gauss-jacobi-cubature-weights}
\end{align}
\end{subequations}

\begin{remark}
For the Hermite, the Laguerre and the Jacobi families the orthogonality relations of
the associated symmetric polynomials $P_\lambda (\mathbf{x})$, $\lambda\in\Lambda^{(n)}$  \eqref{schur}, \eqref{vandermonde}  originating from Proposition \ref{GS-orthogonality:prp}
were pointed out in  Refs. \cite{las:jacobi,las:laguerre,las:hermite}.   In these special cases, Proposition \ref{GS-discrete-orthogonality:prp} now provides 
the complementary discrete
 orthogonality relations underpinning the cubature rules in Eqs. \eqref{gauss-hermite-cubature-rule}, \eqref{gaus-hermite-cubature-weights}, Eqs. \eqref{gauss-laguerre-cubature-rule}, \eqref{gauss-laguerre-cubature-weights}, and Eqs. \eqref{gauss-jacobi-cubature-rule}, \eqref{gauss-jacobi-cubature-weights}.
 \end{remark}

\begin{remark}\label{n=2:rem}
For $n=2$ the bivariate Gaussian cubature rule stemming from Proposition \ref{gauss-cubature:prp} was first formulated in \cite[Equation (1.4)]{sch-xu:bivariate}
(in the symmetrized coordinates $y_1=x_1+x_2$ and $y_2=x_1x_2$ of Remark \ref{coordinate-trafo:rem}).
A more detailed study of the corresponding bivariate Gauss-Jacobi cubature \eqref{gauss-jacobi-cubature-rule}, \eqref{gauss-jacobi-cubature-weights}
 can be found in  \cite{xu:minimal-1,xu:minimal-2}.
\end{remark}

\section{The Bernstein-Szeg\"o polynomials}\label{sec6}
For parameters $\alpha,\beta\in \{ \frac{1}{2}, -\frac{1}{2}\}$, the Gauss-Jacobi cubature \eqref{gauss-jacobi-cubature-rule}, \eqref{gauss-jacobi-cubature-weights}
specializes to more elementary Gauss-Chebyshev cubature rules. For $n=2$ such bivariate Gauss-Chebyshev cubatures were highlighted in \cite{xu:minimal-1,xu:minimal-2} (cf. Remark \ref{n=2:rem} above). For general $n$, a systematic study of closely related Gauss-Chebyshev cubature formulas  was carried out
in  Refs. \cite{moo-pat:cubature,moo-mot-pat:gaussian,hri-mot:discrete,hri-mot-pat:cubature} within the framework of compact simple Lie groups.  From this perspective, the Gauss-Chebyshev cubatures arising here turn out to be associated with the classical Lie groups of  type $B_n$, $C_n$ and $D_n$. 
In the remainder of the paper, we employ the Bernstein-Szeg\"o polynomials \cite[Section 2.6]{sze:orthogonal} to construct   (rational) generalizations of the Gauss-Chebyshev cubatures
stemming from Eqs. \eqref{gauss-jacobi-cubature-rule}, \eqref{gauss-jacobi-cubature-weights} when $\alpha,\beta\in \{ \frac{1}{2}, -\frac{1}{2}\}$.
To this end it will be convenient to pass to trigonometric variables from now on:
\begin{equation*}
\boxed{x=\cos (\xi), \qquad 0\leq \xi\leq \pi.}
\end{equation*}

 By definition (cf.  \cite[Section 2.6]{sze:orthogonal}), the Bernstein-Szeg\"o polynomial $p_l\bigl(\cos ( \xi )\bigr)$ serving our purposes is a polynomial of degree $l$ in $x=\cos( \xi)$ such that
the sequence $p_0\bigl(\cos(\xi)\bigr),p_1\bigl(\cos(\xi)\bigr),p_2\bigl(\cos (\xi)\bigr),\ldots $ provides an orthonormal basis of the Hilbert space $L^2((0,\pi),w(\xi) \text{d}\xi)$, where  the weight function is of the form
\begin{equation}\label{bs-wf}
w (\xi) :=  \frac{{ 2^{\epsilon_+ + \epsilon_- }\bigl(1+\epsilon_+ \cos ( \xi) \bigr)\bigl(1-\epsilon_-\cos (\xi)\bigr)}}{2\pi \prod_{1\leq r\leq d} (1+2a_r\cos (\xi )+a_r^2)}     \qquad (0< |a_r| <1)
\end{equation}
($r=1,\ldots ,d$). Here {$\epsilon_\pm\in\{0,1\}$} and it is moreover assumed (throughout)  that any complex parameters $a_r$ occur in complex conjugate pairs (so ${w}(\xi )$ remains positive and bounded on the interval $(0,\pi)$).

It is well known---cf.  \cite[Section 2.6]{sze:orthogonal}---that 
for $\boxed{{\textstyle l\geq d_\epsilon:= \frac{d-\epsilon_+ - \epsilon_-}{2}}}$, the Bernstein-Szeg\"o polynomial is given by an explicit formula of the form:
\begin{subequations}
\begin{equation}\label{bs-explicit}
p_l(\cos(\xi)) =\Delta_l^{1/2} \Bigl( {c}(\xi)e^{il \xi} + {c}(-\xi)e^{-il\xi}\Bigr) ,
\end{equation}
where
\begin{equation}\label{c-f}
{c}(\xi):=  {(1+\epsilon_+e^{-i\xi})^{-1} (1-\epsilon_-e^{-i\xi})^{-1}}\prod_{1\leq r \leq d} (1+a_r e^{-i\xi})
\end{equation}
(so ${w}(\xi)=1/(2\pi {c}(\xi){c}(-\xi))=1/(2\pi |{c}(\xi)|^{2})$) and
\begin{equation}\label{bs-weights}
\Delta_l  :=  
 \begin{cases} {\bigl(1 + (-1)^{\epsilon_-} \prod_{1\leq r\leq d}a_r\bigr)^{-1}}  &\text{if}\  l={ d_\epsilon} ,\\
 1 &\text{if}\ l > {d_\epsilon}.
 \end{cases} 
 \end{equation}
\end{subequations}

\begin{remark}\label{chebyshev:rem}
For $d=0$ the Bernstein-Szeg\"o polynomials degenerate to
\begin{equation*}
p_l\bigl(\cos (\xi )\bigr)=
\begin{cases}
2^{1-\delta_l/2}\cos(l\xi) &\text{if}\  (\epsilon_+,\epsilon_-)=(0,0) ,\\
\frac{ \epsilon_+\cos((l+\frac12)\xi)\sin(\frac{\xi}2) +\epsilon_- \sin((l+\frac12)\xi)\cos(\frac{\xi}2) }
{(\epsilon_+ +\epsilon_-)\sin(\frac{\xi}2)\cos(\frac{\xi}2)}&\text{if}\ (\epsilon_+,\epsilon_-)\neq (0,0) 
\end{cases}
\end{equation*}
(where $\delta_l:=1$ if $l=0$ and $\delta_l:=0$ otherwise). These are, respectively, the Chebyshev polynomials
of the first kind $ (\epsilon_+,\epsilon_-)=(0,0)$, of the second kind  $ (\epsilon_+,\epsilon_-)=(1,1)$ (so $p_l\bigl(\cos(\xi)\bigr)=\frac{\sin((l+1)\xi)}{\sin(\xi)}$),  of the third kind $ (\epsilon_+,\epsilon_-)=(0,1)$
(so $p_l\bigl(\cos(\xi)\bigr)=\frac{\sin((l+\frac12)\xi)}{\sin(\frac12 \xi)}$), and of the fourth kind
$ (\epsilon_+,\epsilon_-)=(1,0)$ (so $p_l\bigl(\cos(\xi)\bigr)=\frac{\cos((l+\frac12)\xi)}{\cos(\frac12 \xi)}$) (cf. e.g. \cite[Chapter 18]{olv-loz-boi-cla:nist}).
\end{remark}

\begin{remark} The formula in Eqs. \eqref{bs-explicit}--\eqref{bs-weights} is read-off from the following elementary asymptotics  in the complex plane
for ${l\geq d_\epsilon}$:
\begin{subequations}
\begin{equation}\label{asymptotics}
c(\xi)e^{il \xi} + c(-\xi)e^{-il\xi} =    \Delta_l^{-1} e^{il\xi}    +o (e^{il \xi})    \quad \text{as}\ |e^{i\xi}| \to +\infty ,
\end{equation}
in combination with the relatively straightforward  integration formula for $0\leq k\leq l$ (cf. the end of this remark below for some additional indications concerning the evaluation of this integral):
\begin{equation}\label{integral}
\frac{1}{2\pi} \int_{-\pi}^\pi \frac{e^{il \xi}}{c(-\xi)} c_k(\xi) \text{d}\xi =
\begin{cases}
1&\text{if}\ k =l , \\
0&\text{if}\ k <l,
\end{cases}
\end{equation}
\end{subequations}
where $c_k(\xi) := 2^{1-\delta_k} \cos (k\xi)$.
Indeed, since the ({possible}) singularities at $e^{i\xi }=\pm 1$ (stemming from ${c}(\xi)$, {if $\epsilon_++\epsilon_->0$}) are removable in the even expression on the LHS of \eqref{asymptotics}, 
it is clear from the asymptotics that we are dealing with a polynomial of degree $l(\geq d_\epsilon)$ in $\cos (\xi)$.
Moreover, it follows from
Eq. \eqref{integral} that
\begin{equation}
 \int_0^\pi   \bigl( {c}(\xi )e^{il\xi}+\text{c}(-\xi)e^{-il\xi}\bigr) c_k(\xi) {w}(\xi ) \text{d} \xi
=
\begin{cases}
1&\text{if}\ k =l , \\
0&\text{if}\ k <l 
\end{cases}
\end{equation}
(where the overall numerical factor is absorbed in the weight function $w(\xi)$).
The upshot is that for $l\geq {d_\epsilon}$ the RHS of \eqref{bs-explicit} satisfies the defining
orthogonality relations for $p_l\bigl(\cos (\xi)\bigr)$. Notice that this also reveals that  the (leading) coefficient $\alpha_l$ of $\bigl( \cos(\xi)\bigr)^l$ in $p_l\bigl(\cos (\xi)\bigr)$ is given by
\begin{equation}\label{lead-coef}
\alpha_l=2^l\Delta_l^{-1/2}
\end{equation}
in this situation. Finally, to infer the identity in Eq. \eqref{integral} it suffices to observe that the integral under consideration picks up the constant term of the (Fourier) expansion  in $e^{i\xi}$ of the integrand. Indeed, after expanding the $d$ factors stemming from the denominator of
$1/c(-\xi)$ in terms of  geometric series, it readily follows that the constant term in question is equal to $0$ if $l>k\geq 0$ and equal to $1$ if $l=k\geq 0$.
\end{remark}

\section{On the roots of Bernstein-Szeg\"o polynomials}\label{sec7}

\begin{subequations}
For $m+1\geq {d_\epsilon}$, the explicit representation in Eqs. \eqref{bs-explicit}--\eqref{bs-weights}
permits to compute the
 $(\hat{l}+1)$th root  $\xi_{\hat{l}}^{(m+1)}$ of the Bernstein-Szeg\"o polynomial
 \begin{equation}\label{bs-factorization}
p_{m+1}\bigl(\cos (\xi)\bigr) = \alpha_{m+1} \prod_{0\leq \hat{l}\leq m}   \left(  \cos (\xi)-\cos (\xi_{\hat{l}}^{(m+1)}) \right) ,
\end{equation}
with the convention
 \begin{equation}\label{bs-roots}
0< \xi_0^{(m+1)}< \xi_1^{(m+1)}<\cdots <  \xi_{m}^{(m+1)}< \pi,
\end{equation}
\end{subequations}
 as the unique real solution of an elementary transcendental equation.

\begin{proposition}[Bernstein-Szeg\"o  Roots]\label{bs-roots:prp}
Given $m+1\geq {d_\epsilon}$ and $\hat{l}\in \{ 0,\ldots ,m\}$, the root  $\xi_{\hat{l}}^{(m+1)}$ \eqref{bs-roots} of the Bernstein-Szeg\"o polynomial
 $p_{m+1}\bigl(\cos (\xi)\bigr)$ can be retrieved as the unique real solution of the transcendental equation
\begin{subequations}
\begin{equation}\label{bethe-eq}
{2(m+1-d_\epsilon)}\xi + \sum_{1\leq r\leq d} v_{a_r}(\xi)   =
 {\pi \left( 2\hat{l}+{1+\epsilon_-} \right)}
\end{equation}
where
\begin{equation}\label{va}
v_a(\xi):=  \int_0^\xi    \frac{1-a^2}{1+2a\cos (\theta )+a^2}\text{d} \theta \qquad (|a|< 1) .
\end{equation}
\end{subequations}
\end{proposition}

\begin{proof}
Since for $|a|<1$ and $\xi$ real
$v_a^\prime(\xi)+v_{\bar{a}}^\prime (\xi) >0$, it is clear that
the (odd) real function of $\xi$ on
the LHS of Eq. \eqref{bethe-eq} is monotonously increasing  and unbounded (as $v_a(\xi+2\pi)=v_a(\xi)+2\pi$). The transcendental equation in question has therefore
a unique real solution $\hat{\xi}_{\hat{l}}^{(m+1)}$ (say).
Moreover, from the RHS  (and the monotonicity of the LHS) we see that $\hat{\xi}_{\hat{k}}^{(m+1)}> \hat{\xi}_{\hat{l}}^{(m+1)}$ if $\hat{k}>\hat{l}$. At $\xi=0$ and $\xi=\pi$ the LHS of Eq. \eqref{bethe-eq}  takes the values $0$ and ${(2m+2+\epsilon_+ +\epsilon_-)\pi}$, respectively (because $v_a(\pi)=\pi$),
so it is clear (by comparing with the values on the RHS) that $0 <\hat{\xi}_0^{(m+1)}< \hat{\xi}_1^{(m+1)}<\cdots < \hat{\xi}_m^{(m+1)}<\pi$.
It remains to infer that at $\xi=\hat{\xi}_{\hat{l}}^{(m+1)}$ ($0\leq \hat{l}\leq m$) our Bernstein-Szeg\"o polynomial
$p_{m+1}\bigl( \cos (\xi)\bigr)$ vanishes, or equivalently (when $m+1\geq {d_\epsilon}$) that
$e^{2i(m+1)\xi}=-\frac{{c}(-\xi)}{{c}(\xi)}$, or more explicitly:
\begin{equation}\label{bethe}
e^{{2i(m+1-d_\epsilon)}\xi} ={(-1)^{\epsilon_-+1}}\prod_{1\leq r\leq d} \frac{1+a_r e^{i\xi}}{e^{i\xi}+a_r} .
\end{equation}
Multiplication of Eq. \eqref{bethe-eq} by $i$ and exponentiation of both sides with the aid of the identity (cf. Eq. \eqref{va:elementary} below)
\begin{equation*}
e^{-iv_a(\xi)}=  \frac{1+a e^{i\xi}}{e^{i\xi}+a}  \qquad (|a|<1),
\end{equation*}
reveals that Eq. \eqref{bethe} is automatically satisfied at solutions of Eq. \eqref{bethe-eq}, i.e. $p_{m+1}(\hat{\xi}_{\hat{l}}^{(m+1)})=0$ and thus $\hat{\xi}_{\hat{l}}^{(m+1)}=\xi_{\hat{l}}^{(m+1)}$ (for $\hat{l}=0,\ldots ,m$).
\end{proof}

 Proposition \ref{bs-roots:prp} entails the following estimates for the Bernstein-Szeg\"o roots and their distances.

\begin{proposition}[Estimates for the Bernstein-Szeg\"o  Roots]\label{bs-bounds:prp}
For $m+1\geq {d_\epsilon}$, the Bernstein-Szeg\"o roots \eqref{bs-roots} obey the following inequalities:
\begin{subequations}
\begin{equation}\label{bound:a}
\frac{\pi \bigl(\hat{l}+{\frac12+\frac{\epsilon_-}{2}\bigr)}}{m+1-{d_\epsilon}+\kappa_-}   \leq \xi_{\hat{l}}^{(m+1)}  \leq \frac{\pi \bigl(\hat{l}+ {\frac12+\frac{\epsilon_-}{2}\bigr)}}{m+1-{d_\epsilon}+\kappa_+} 
\end{equation}
(for $0\leq \hat{l}\leq m$), and
\begin{equation}\label{bound:b}
\frac{\pi (\hat{k}-\hat{l})}{m+1-{d_\epsilon}+\kappa_-}   \leq \xi_{\hat{k}}^{(m+1)} - \xi_{\hat{l}}^{(m+1)} \leq \frac{\pi (\hat{k}-\hat{l})}{m+1-{d_\epsilon}+\kappa_+} 
\end{equation}
(for $0\leq \hat{l}<\hat{k}\leq m$), where
\begin{equation}
\kappa_\pm := \frac{1}{2}\sum_{1\leq r\leq d}   \left(\frac{1- |a_r|}{1 + |a_r|}\right)^{\pm 1} .
\end{equation}
\end{subequations}
\end{proposition}

\begin{proof}
The estimate in Eq. \eqref{bound:a} readily follows from the transcendental equation for $\xi_{\hat{l}}^{(m+1)}$ in Eq. \eqref{bethe-eq} through the mean value theorem. Here one uses that for $\xi$ real
\begin{equation*}
 \text{Re} \bigl(v_a^\prime (\xi)\bigr) =  \frac{1}{2}  \Bigl(    v_{|a|}^\prime \bigl(\xi+\text{Arg}(a)\bigr)+ v_{|a|}^\prime \bigl(\xi-\text{Arg}(a)\bigr) \Bigr) ,
\end{equation*}
whence
\begin{equation*}
\frac{1-|a|}{1 + |a|} \leq \text{Re} \bigl(v_a^\prime (\xi)\bigr) \leq    \frac{1+|a|}{1 - |a|}  \qquad (|a|<1).
\end{equation*}
The estimate in Eq. \eqref{bound:b}  for the distance between the zeros follows in an analogous way, after subtracting the 
$\hat{l}$th equation in Eq. \eqref{bethe-eq} from the $\hat{k}$th equation.
\end{proof}

\begin{remark}\label{numerics:rem}
The
transcendental equation in Proposition \ref{bs-roots:prp} is well-suited for computing $\xi_{\hat{l}}^{(m+1)}$ ($m+1\geq {d_\epsilon}$)
numerically (e.g. by means of a standard  fixed-point iteration scheme like Newton's method).
 Notice in this connection  that
 for $-\pi< \xi<\pi$ (and $|a|<1$):
\begin{equation}\label{va:elementary}
v_a(\xi)=i \text{Log}
\biggl( \frac{1+ ae^{i\xi}}{e^{i\xi} +a }  \biggr) =  2\text{Arctan} \Biggl(    \frac{1-a}{1+a}\tan \left( \frac{\xi}{2} \right)   \Biggr),
\end{equation}
so numerical integration can be readily avoided when evaluating $v_a(\xi)$ \eqref{va}. A natural initial estimate for starting up the numerical computation of $\xi_{\hat{l}}^{(m+1)}$ is provided by
the exact $(\hat{l}+1)$th Chebyshev root  
$
{
\frac
{\pi \bigl(\hat{l}+\frac{1}{2}+\frac{\epsilon_-}{2}  \bigr)}
{m+1+\frac{\epsilon_+}2+\frac{\epsilon_-}2}
}
$
(which corresponds to the case $d=0$, cf. Remark \ref{chebyshev:rem}). Indeed, these Chebyshev roots automatically comply with all inequalities in Proposition \ref{bs-bounds:prp}.
\end{remark}

\section{Gauss-Chebyshev quadrature for rational functions with prescribed poles}\label{sec8}
For $\epsilon_\pm=0$ compact expressions for the Christoffel weights associated with the Bernstein-Szeg\"o polynomials were computed in \cite[Theorem 4.4]{dar-gon-jim:quadrature},
while for  general $\epsilon_\pm\in \{ 0,1\}$ the corresponding formulas can be gleaned from
 \cite[Theorems 5.3--5.5]{bul-cru-dec-gon:rational}:
 \begin{align}
{{w}}^{(m+1)}_{\hat{l}} := & \Big( \sum_{0\leq l\leq m} p^2_l\bigl( \cos(\xi_{\hat{l}}^{(m+1)})\bigr) \Bigr)^{-1} \nonumber \\
=& \Bigl(   |c (\xi_{\hat{l}}^{(m+1)}) |^2  h^{(m+1)} (\xi_{\hat{l}}^{(m+1)}) \Bigr)^{-1} \label{christoffel-weights} \\
&\qquad\text{with}\quad  {h}^{(m+1)}(\xi ):= 
2(m +1-{d_\epsilon})+\sum_{1\leq r\leq d} v^\prime_{a_r}(\xi ) \nonumber
\end{align}
$ (\hat{l}=0,\ldots ,m)$, where it was assumed that $m+1\geq d_\epsilon$.   To keep our presentation self-contained, a short verification of Eq. \eqref{christoffel-weights} is provided in Appendix \ref{appA} below.

The Gauss quadrature \eqref{gqr:a}, \eqref{gqr:b} now gives rise to the following exact quadrature rule for the integration of rational functions with prescribed poles
against the Chebyshev  weight function (cf. \cite[Section 4]{dar-gon-jim:quadrature} and \cite[Section 5]{bul-cru-dec-gon:rational}):
\begin{subequations}
\begin{align}\label{gc-rule:a}
\frac{1}{2\pi} \int_0^\pi  R (\xi)  \rho  (\xi) \text{d}\xi =
 \sum_{0\leq \hat{l}\leq m}   R\bigl( \xi _{\hat{l}}^{(m+1)}\bigr)  \rho \bigl(\xi _{\hat{l}}^{(m+1)}\bigr) \Bigl( h^{(m+1)} \bigl(\xi_{\hat{l}}^{(m+1)}\bigr)\Bigr)^{-1}
.
\end{align}
Here $\rho (\cdot )$ refers to the Chebyshev weight function
\begin{equation}\label{gc-rule:b}
\rho (\xi):=2^{\epsilon_++\epsilon_-}(1+\epsilon_+\cos(\xi))   (1-\epsilon_-\cos(\xi))
\end{equation}
and $R(\cdot)$ is of the form
\begin{equation}\label{gc-rule:c}
R(\xi )  =   \frac{f\bigl(\cos (\xi)\bigr)}{\prod_{1\leq r\leq d} \bigl(1+2a_r\cos (\xi )+a_r^2\bigr)}   
\end{equation}
\end{subequations}
with $d\leq 2(m+1)+\epsilon_++\epsilon_-$, where $f(\cos(\xi))$ denotes an arbitrary polynomial  of degree at most $2m+1$ in $\cos (\xi)$. For $d=0$, the quadrature rule in Eqs. \eqref{gc-rule:a}--\eqref{gc-rule:c} reproduces the standard Gauss-Chebyshev quadratures (cf. Remark \ref{chebyshev:rem}).

\begin{remark}\label{bs-orthogonality:rem}
Assuming $m+1\geq d_\epsilon$,
the underlying discrete orthogonality relations for the Bernstein-Szeg\"o polynomials (cf. Eqs.  \eqref{finite-orthogonality} and \eqref{finite-orthogonality:dual}) become explicitly
\begin{subequations}
\begin{align}\label{bs:orthogonality}
  \sum_{0\leq \hat{l}\leq m}    p_l\bigl( \cos(\xi_{\hat{l}}^{(m+1)})\bigr)  p_{k} \bigl( \cos(\xi_{\hat{l}}^{(m+1)})\bigr) & \Bigl(  |c (\xi_{\hat{l}}^{(m+1)}) |^2  h^{(m+1)} (\xi_{\hat{l}}^{(m+1)}) \Bigr)^{-1}
  \\
 &=
\begin{cases}
1&\text{if}\ k =l, \\
0&\text{if}\ k \neq l
\end{cases} \nonumber
\end{align}
($l,k \in \{ 0,\ldots ,m\}$) and
\begin{align}\label{bsd:orthogonality}
  \sum_{0\leq l\leq m}    p_l\bigl( \cos(\xi_{\hat{l}}^{(m+1)})\bigr) p_{l} & \bigl( \cos (\xi_{\hat{k}}^{(m+1)}) \bigr) \\
  &=
\begin{cases}
  |c (\xi_{\hat{l}}^{(m+1)}) |^2  h^{(m+1)} (\xi_{\hat{l}}^{(m+1)})  &\text{if}\ \hat{k} =\hat{l}, \\
0&\text{if}\ \hat{k} \neq \hat{l}
\end{cases} \nonumber
\end{align}
\end{subequations}
($\hat{l},\hat{k} \in \{ 0,\ldots ,m\}$), respectively.
\end{remark}

\section{Gauss-Chebyshev cubature for symmetric rational functions with prescribed poles at coordinate hyperplanes}\label{sec9}
The specialization of Proposition \ref{gauss-cubature:prp} to the case of the Bernstein-Szeg\"o polynomials now immediately  culminates in the principal result of this paper: 
an explicit cubature rule for the integration 
of symmetric functions---with prescribed poles at coordinate hyperplanes---against the distributions of the unitary circular Jacobi ensembles.
The cubature in question
generalizes the quadrature in Eqs. \eqref{gc-rule:a}--\eqref{gc-rule:c} to the situation of an arbitrary number of variables $n\geq 1$.

\begin{theorem}[Gauss-Chebyshev Cubature Rule for Symmetric Functions]\label{gauss-chebyshev-cubature:thm}
Let $\epsilon_\pm\in \{ 0, 1\}$ and $|a_r|<1$ ($r=1,\ldots ,d$) with (possible) complex parameters $a_r$ arising in complex conjugate pairs.
Then assuming
\begin{equation*}
d\leq 2(m+n)+\epsilon_+ +\epsilon_-, 
\end{equation*}
one has that 
\begin{subequations}
\begin{align}\label{gauss-chebyshev-cubature:a}
\frac{1}{(2\pi )^n\, n!}&\int_0^\pi\cdots \int_0^\pi  R ( \boldsymbol{\xi})  \rho  (\boldsymbol{\xi}) \text{d}\xi_1\cdots \text{d}\xi_n =\\
& \sum_{\hat{\lambda}\in\Lambda^{(m,n)}}   R\bigl( \boldsymbol{ \xi} _{\hat{\lambda}}^{(m,n)} \bigr)  \rho \bigl(\boldsymbol{ \xi} _{\hat{\lambda}}^{(m,n)} \bigr) \Bigl( H^{(m,n)} \bigl(\boldsymbol{ \xi} _{\hat{\lambda}}^{(m,n)}\bigr)\Bigr)^{-1} 
.\nonumber
\end{align}
Here the nodes $ \boldsymbol{ \xi} _{\hat{\lambda}}^{(m,n)}$ are of the form in Eq. \eqref{X-nodes} with $\xi^{(m+1)}_{\hat{l}}$ as in Eqs. \eqref{bs-factorization}, \eqref{bs-roots} (cf. also Propositions \ref{bs-roots:prp}, \ref{bs-bounds:prp}),  the weight function $\rho  (\cdot )$ refers to the unitary circular Jacobi distribution
\begin{align}\label{rho-jacobi}
\rho  (\boldsymbol{\xi}):=\prod_{1\leq j\leq n} 2^{\epsilon_++\epsilon_-} & \bigl(1+\epsilon_+\cos(\xi_j)\bigr)   \bigl(1-\epsilon_-\cos(\xi_j)\bigr) \\
&\times \prod_{1\leq j<k\leq n}  \bigl(\cos (\xi_j)-\cos(\xi_k)\bigr)^2 ,\nonumber
\end{align}
the Christoffel weights are governed by
\begin{equation}
H^{(m,n)}(\boldsymbol{\xi}) := \prod_{1\leq j\leq n}   h^{(m+n)}(\xi_j)
\end{equation}
with $h^{(m+1)}(\cdot )$ taken from Eq. \eqref{christoffel-weights}, and $R(\cdot)$ is of the form
\begin{equation}
R(\boldsymbol{\xi})  =   \frac{f\bigl(\cos (\xi_1),\ldots ,\cos(\xi_n)\bigr)}{\prod_{\substack{1\leq r\leq d\\ 1\leq j\leq n}} \bigl(1+2a_r\cos (\xi_j )+a_r^2\bigr)}    ,
\end{equation}
\end{subequations}
 where $f(x_1,\ldots,x_n)=f(\mathbf{x})$ denotes an arbitrary symmetric polynomial in $\mathbb{P}^{(2m+1,n)}$.
\end{theorem}

When $d=0$, Theorem \ref{gauss-chebyshev-cubature:thm} reduces to a
Gauss-Chebyshev cubature of the form (cf. Remark \ref{chebyshev:rem})
\begin{subequations}
\begin{align}\label{GCc:a}
\frac{1}{(2\pi )^n\, n!}&\int_0^\pi\cdots \int_0^\pi  f \bigl( \cos(\boldsymbol{\xi})\bigr)  \rho  (\boldsymbol{\xi}) \text{d}\xi_1\cdots \text{d}\xi_n =\\
& \frac{1}{\left( 2(m+n)+\epsilon_+ +\epsilon_-\right)^n}
\sum_{\hat{\lambda}\in\Lambda^{(m,n)}}   f \left( \cos \bigl( \boldsymbol{ \xi} _{\hat{\lambda}}^{(m,n)} \bigr) \right) \rho \bigl(\boldsymbol{ \xi} _{\hat{\lambda}}^{(m,n)} \bigr) ,\nonumber
\end{align}
where $ f \bigl( \cos(\boldsymbol{\xi})\bigr) := f\bigl(\cos(\xi_1),\ldots,\cos(\xi_n)\bigr)$ with $f(x_1,\ldots,x_n)=f(\mathbf{x})\in \mathbb{P}^{(2m+1,n)}$, and
with explicit nodes $ \boldsymbol{ \xi} _{\hat{\lambda}}^{(m,n)}$ \eqref{X-nodes} governed by the Chebyshev roots:
\begin{equation}\label{GCc:b}
\xi^{(m+n)}_{\hat{l}}=
\frac
{\pi \bigl(\hat{l}+\frac{1}{2}+\frac{\epsilon_-}{2}  \bigr)}
{m+n+\frac{\epsilon_+}2+\frac{\epsilon_-}2} \qquad (0\le \hat{l}<m+n).
\end{equation}
\end{subequations}
The latter  cubature formula turns out to be closely related to a class of integration rules of Lie-theoretic nature studied
 Refs. \cite{moo-pat:cubature,moo-mot-pat:gaussian,hri-mot:discrete,hri-mot-pat:cubature}, upon restricting to the classical simple Lie groups of type $B_n$, $C_n$ and $D_n$ (cf. Remarks \ref{haar:rem} and \ref{moody-patera:rem} below for some further details).

 \begin{remark}\label{f-alcove:rem}
 By exploiting the symmetry of the integrand in the coordinates, the LHS of Eq. \eqref{gauss-chebyshev-cubature:a} can be
 readily rewritten as a multivariate integral over the fundamental domain
\begin{equation}\label{open-alcove}
\mathbb{A}^{(n)}:=\{  \boldsymbol{\xi}\in\mathbb{R}^n\mid \pi > \xi_1>\xi_2>\cdots >\xi_n>0\}  .
\end{equation}
The estimates in Proposition \ref{bs-bounds:prp} then ensure that the cubature nodes $\xi^{(m,n)}_{\hat{\lambda}}$ , $\hat{\lambda}\in\Lambda^{(m,n)}$ lie inside
this ordered domain of integration. 
 \end{remark}

\begin{remark}\label{haar:rem}
The unitary circular Jacobi distributions $\rho (\cdot )$ \eqref{rho-jacobi} correspond  to the
Haar measures on the compact simple Lie groups $SO(2n+1;\mathbb{R})$
(type $B_n$: $\epsilon_+ \neq \epsilon_-$), $Sp (n;\mathbb{H})$ (type $C_n$: $\epsilon_\pm =1$), and $SO(2n;\mathbb{R})$ (type $D_n$: $\epsilon_\pm =0$), cf. e.g.
 \cite[Chapter IX.9]{sim:representations}, \cite[Chapter 11.10]{pro:lie} or \cite[Chapter 2.6]{for:log-gases}.
 \end{remark}

 \begin{remark}\label{moody-patera:rem}
 The Gauss-Chebyshev  cubature formula in Eqs. \eqref{GCc:a}, \eqref{GCc:b} should be viewed as a counterpart pertaining to the \emph{nonreduced} root system $R=BC_n$ of the cubature rules in
 \cite[Theorem 7.2]{moo-pat:cubature} (where the situation of  \emph{reduced} crystallographic  root systems was considered). The choice for the kind of underlying Chebyshev polynomials originates in this perspective from a freedom in the weight function $\rho (\boldsymbol{\xi})$ \eqref{rho-jacobi}, which is
given (up to normalization) by the squared modulus of the Weyl denominator of one of the following three reduced subsystems of $R=BC_n$:  type $B_n$ ($\epsilon_+ \neq \epsilon_-$),  type $C_n$ ($\epsilon_\pm =1$), or  type $D_n$ ($\epsilon_\pm =0$), respectively (cf. Remarks \ref{chebyshev:rem} and \ref{haar:rem} above). For $\epsilon_-=0$, the cubature in Eqs. \eqref{GCc:a}, \eqref{GCc:b} can actually already be retrieved from \cite[Theorem 5.2]{hri-mot:discrete} (second part). Notice in this connection that both in \cite{moo-pat:cubature} and in \cite{hri-mot:discrete} the corresponding cubature rules are formulated in symmetrized coordinates involving  transformations analogous to the one in Remark \ref{coordinate-trafo:rem}.
To further facilitate the comparison of Eqs. \eqref{GCc:a}, \eqref{GCc:b}  with the formulas in  \cite[Theorem 7.2]{moo-pat:cubature}, let us briefly
recall that the Weyl group of the root system $BC_n$ is given by the hyperoctahedral group of signed permutations (acting on the coordinates $\xi_1,\ldots ,\xi_n$ through permutations and sign flips).
The closure of the ordered integration domain $\mathbb{A}^{(n)}$ \eqref{open-alcove}, which constitutes a fundamental domain
for this Weyl group action  on $\mathbb{T}^{(n)}:= \mathbb{R}^n/ (2\pi \mathbb{Z})^n$, coincides  (up to rescaling) with the positive Weyl alcove of the root system.  Similarly, the 
set
$\Lambda^{(m,n)}$ \eqref{f-alcove}, which labels both the Schur (character) basis of $\mathbb{P}^{(m,n)}$ and the cubature nodes $\xi^{(m,n)}_{\hat{\lambda}}$, arises as a fundamental domain
for the Weyl group action on $ \mathbb{Z}^n/ (2m \mathbb{Z})^n$.  This fundamental domain is built of the $BC_n$ root system's dominant weights
of the form
\begin{equation}
\lambda =  l_1\varpi_1+\cdots +l_n\varpi_n   \quad \text{with}\quad l_1,\ldots ,l_n\geq 0\ \text{and}\  l_1+\cdots +l_n\leq m,
\end{equation}
where $\varpi_k:=e_1+\cdots +e_k$ ($k=1,\ldots,n$) refers to the basis of the fundamental weights (and $e_1,\ldots ,e_n$ denotes the standard unit basis of (the weight lattice) $\mathbb{Z}^n$).
The fact that in the present situation both the (Schur) character basis and the cubature nodes are labeled by the same dominant weights reflects the self-duality of the root system $BC_n$, whereas in general one has to resort to both weights (for labeling the character basis) and coweights (for labeling the nodes) \cite{moo-pat:cubature}.
\end{remark}

\begin{remark}\label{GS-bs-orthogonality:rem}
For $\lambda\in\Lambda^{(n)}$, let $P_\lambda \bigl(\cos (\boldsymbol{\xi})\bigr):=P_\lambda \bigl(\cos (\xi_1),\ldots,\cos(\xi_n)\bigr)$ denote the generalized Schur polynomial from Eqs. \eqref{schur}, \eqref{vandermonde} associated with the orthonormal Bernstein-Szeg\"o family $p_l\bigl(\cos (\xi)\bigr)$ from Section \ref{sec6}.
The orthogonality relations from Proposition \ref{GS-orthogonality:prp} now become: 
\begin{subequations}
\begin{align}
\frac{1}{(2\pi )^n\, n!} \int_0^\pi\cdots\int_0^\pi P_\lambda \bigl(\cos (\boldsymbol{\xi})\bigr) P_\mu  \bigl(\cos (\boldsymbol{\xi})\bigr)  | C(\boldsymbol{\xi}) |^{-2}
\text{d}\xi_1\cdots\text{d}\xi_n & \\
= \begin{cases}
1&\text{if}\ \mu =\lambda, \\
0&\text{if}\  \mu\neq\lambda 
\end{cases} & \nonumber
\end{align}
($\lambda,\mu\in \Lambda^{(n)}$), where
\begin{equation}\label{C-f}
C(\boldsymbol{\xi}):= 2^{n(n-1)/2} \prod_{1\leq j\leq n} c(\xi_j) \prod_{1\leq j<k\leq n}    (1-e^{-i(\xi_j+\xi_k)})^{-1} (1-e^{-i(\xi_j-\xi_k)})^{-1}
\end{equation}
\end{subequations}
(so $| C(\boldsymbol{\xi} )|^{-2}= \prod_{1\leq j\leq n} |c(\xi_j)|^{-2} \prod_{1\leq j<k\leq n}\bigl( \cos(\xi_j)-\cos(\xi_k)\bigr)^2$) with $c(\xi)$ taken from Eq. \eqref{c-f}.
Upon expanding the pertinent determinant from $P_\lambda$ \eqref{schur}, \eqref{vandermonde}, one arrives at a
multivariate generalization of the explicit formula in Eqs. \eqref{bs-explicit}--\eqref{bs-weights} that is valid for
$\lambda\in\Lambda^{(n)}$ with $\lambda_n\geq d_\epsilon$:
\begin{subequations}
\begin{align}
P_\lambda & \bigl(\cos (\boldsymbol{\xi})\bigr)=  \\
& \Delta_\lambda^{1/2} \sum_{\substack{ \varepsilon_1,\ldots,\varepsilon_n\in \{ -1 ,1\} \\ \sigma \in S_n}}
C(\varepsilon_1\xi_{\sigma_1},\ldots ,\varepsilon_n\xi_{\sigma_n})  \exp \bigl( i\varepsilon_1\lambda_1 \xi_{\sigma_1}+\cdots+ i\varepsilon_n\lambda_n \xi_{\sigma_n}\bigr) ,
\nonumber
\end{align}
where
\begin{equation}
\Delta_\lambda := \prod_{1\leq j \leq n}   \Delta_{\lambda_j+n-j}=\Delta_{\lambda_n}
\end{equation}
\end{subequations}
and $C(\xi_1,\ldots,\xi_n)=C(\boldsymbol{\xi})$.
This explicit formula reveals that the polynomials in question are  special instances of the multivariate Bernstein-Szeg\"o polynomials associated with root systems appearing in
\cite{die:asymptotics} (for nonreduced root systems) and in \cite{die-maz-ryo:berstein-szego} (for reduced root systems). 
\end{remark}

\begin{remark}
In the situation of the previous remark,
the orthogonality relations of Proposition \ref{GS-discrete-orthogonality:prp} and Remark \ref{GS-dual-discrete-orthogonality:rem} give rise to
the following multivariate generalization
of the discrete orthogonality relations for the Bernstein-Szeg\"o polynomials in Remark \ref{bs-orthogonality:rem} when $m+n\geq d_\epsilon$:
\begin{subequations}
\begin{align}\label{GS-bs:orthogonality}
  \sum_{\hat{\lambda}\in\Lambda^{(m,n)}}    P_\lambda \bigl( \cos(\boldsymbol{\xi}_{\hat{\lambda}}^{(m,n)})\bigr) P_{\mu} \bigl( \cos(\boldsymbol{\xi}_{\hat{\lambda}}^{(m,n)})\bigr)  \Bigl(  |C (\boldsymbol{\xi}_{\hat{\lambda}}^{(m,n)}) |^2  H^{(m,n)} (\boldsymbol{\xi}_{\hat{\lambda}}^{(m,n)}) \Bigr)^{-1}&\\ 
= \begin{cases}
1&\text{if}\ \mu =\lambda, \\
0&\text{if}\  \mu\neq\lambda 
\end{cases} & \nonumber
\end{align}
($\lambda,\mu\in \Lambda^{(m,n)}$) and
\begin{align}\label{GS-bsd:orthogonality}
  \sum_{\lambda\in\Lambda^{(m,n)}}    P_\lambda \bigl( \cos(\boldsymbol{\xi}_{\hat{\lambda}}^{(m,n)})\bigr) & P_\lambda \bigl( \cos (\boldsymbol{\xi}_{\hat{\mu}}^{(m,n)}) \bigr)
\\
&=\begin{cases}
  |C (\boldsymbol{\xi}_{\hat{\lambda}}^{(m,n)}) |^2  H^{(m,n)} (\boldsymbol{\xi}_{\hat{\lambda}}^{(m,n)})  &\text{if}\ \hat{\mu} =\hat{\lambda}, \\
0&\text{if}\ \hat{\mu} \neq \hat{\lambda}
\end{cases} \nonumber
\end{align}
\end{subequations}
($\hat{\lambda},\hat{\mu} \in \Lambda^{(m,n)}$),  respectively.  (Here
the parameter restrictions and the notations are in accordance with Theorem \ref{gauss-chebyshev-cubature:thm} and  Remark \ref{GS-bs-orthogonality:rem}.)
\end{remark}

\section*{Acknowledgments}
We are grateful to the referees for their helpful constructive remarks and, in particular, for pointing out the relevance of the seminal work of Berens, Schmid and Xu \cite{ber-sch-xu:multivariate}.

\appendix

\section{Explicit Christoffel weights for the Gauss quadrature associated with Bernstein-Szeg\"o polynomials}\label{appA}
In this appendix we provide a short verification of Eq. \eqref{christoffel-weights} based on the Christoffel-Darboux kernel
\begin{subequations}
\begin{equation}
\sum_{0\leq l\leq m} p_l(x)p_l(y)=   \frac{\alpha_m}{\alpha_{m+1}}   \frac{p_{m+1}(x)p_m(y)-p_m(x)p_{m+1}(y)}{x-y}
\end{equation}
on the diagonal $y\to x$
\begin{equation}\label{diagonal}
\sum_{0\leq l\leq m} \bigl(p_l(x)\bigr)^2 =   \frac{\alpha_m}{\alpha_{m+1}}  \Bigl(   p_{m+1}^\prime (x)p_m(x) -p_{m+1}(x)p_m^\prime(x)\Bigl).
\end{equation}
\end{subequations}
At the root $x=\cos(\xi^{(m+1)}_{\hat{l}})$ of $p_{m+1}(x)$ Eq. \eqref{diagonal}  produces (cf. e.g.  \cite[Equation (3.4.7)]{sze:orthogonal}):
\begin{equation}\label{cw}
{{w}}^{(m+1)}_{\hat{l}}=  \frac{-\alpha_{m+2}/\alpha_{m+1}}{p_{m+2}\bigl(  \cos (\xi^{(m+1)}_{\hat{l}})\bigr) p_{m+1}^\prime \bigl(  \cos (\xi^{(m+1)}_{\hat{l}})\bigr) }\qquad (0\leq \hat{l}\leq m) ,
\end{equation}
where we have used that $\alpha_{m}\alpha_{m+2}p_m\bigl(\cos(\xi^{(m+1)}_{\hat{l}})\bigr)=-\alpha_{m+1}^2 p_{m+2}\bigl(\cos(\xi^{(m+1)}_{\hat{l}})\bigr)$ (by the three-term recurrence relation).
Combined with the explicit expressions  for the Bernstein-Szeg\"o polynomials in Eqs. \eqref{bs-explicit}--\eqref{bs-weights} and Eq. \eqref{lead-coef} for $l\geq m+1\geq d_\epsilon$, the formula in Eq. \eqref{cw} readily produces Eq. \eqref{christoffel-weights}  upon invoking that  $p_{m+1}\bigl(\cos (\xi^{(m+1)}_{\hat{l}})\bigr)=0$, i.e.
\begin{equation}\label{bae}
e^{2i(m+1)\xi}=-\frac{c(-\xi)}{c(\xi)}\quad  \text{at}   \quad \xi=\xi^{(m+1)}_{\hat{l}}
\end{equation}
(cf. the proof of Proposition \ref{bs-bounds:prp}). Indeed, we read-off from Eqs. \eqref{bs-explicit}--\eqref{bs-weights} that
\begin{align}\label{pm+2}
p_{m+2}\bigl(  \cos (\xi )\bigr) = &c(-\xi)e^{-i(m+1)\xi} \left( \frac{c(\xi)}{c(-\xi)}e^{i(2m+3)\xi}+e^{-i\xi}       \right) \\
\stackrel{\text{Eq.}~\eqref{bae}}{=}&  -2i c(-\xi) e^{-i(m+1)\xi} \sin (\xi) \nonumber 
\end{align}
and that
\begin{align}\label{p'm+1}
&p_{m+1}^\prime \bigl(  \cos (\xi )\bigr)=-\bigl( \sin (\xi)\bigr)^{-1} \Delta^{1/2}_{m+1} c(\xi) e^{i(m+1)\xi} \times \\
& \left(  \frac{c^\prime (\xi)}{c(\xi)}  - \frac{c^\prime (-\xi)}{c(\xi)}e^{-2i(m+1)\xi}+i(m+1)-i(m+1)\frac{c(-\xi)}{c(\xi)}e^{-2i(m+1)\xi} \right) \nonumber \\
&\stackrel{\text{Eq.}~\eqref{bae}}{=}  
-i\bigl( \sin (\xi)\bigr)^{-1} \Delta^{1/2}_{m+1} c(\xi) e^{i(m+1)\xi} \left( 2(m+1) +\frac{1}{i}\left( \frac{c^\prime (\xi)}{c(\xi)}  + \frac{c^\prime (-\xi)}{c(-\xi)} \right) \right) . \nonumber
\end{align}
Substition of Eqs. \eqref{pm+2}, \eqref{p'm+1} and Eq. \eqref{lead-coef} into Eq. \eqref{cw} entails that
\begin{equation}\nonumber
{{w}}^{(m+1)}_{\hat{l}} =  \frac{1}{c ( \xi^{(m+1)}_{\hat{l}})  c(- \xi^{(m+1)}_{\hat{l}}) } \left(   2(m+1)    +\frac{1}{i}\left( \frac{c^\prime (\xi^{(m+1)}_{\hat{l}})}{c(\xi^{(m+1)}_{\hat{l}})}  + \frac{c^\prime (-\xi^{(m+1)}_{\hat{l}})}{c(-\xi^{(m+1)}_{\hat{l}})} \right)    \right)^{-1} .
\end{equation}
Eq. \eqref{christoffel-weights} now follows upon making (the imaginary part of) the logarithmic derivative of $c(\xi)$ \eqref{c-f} explicit:
\begin{equation}
\frac{1}{i}\left( \frac{c^\prime (\xi)}{c(\xi)}  + \frac{c^\prime (-\xi)}{c(-\xi)} \right)=
 \epsilon_++ \epsilon_-+ \sum_{1\leq r\leq d}   \left( \frac{1-a_r^2}{1+a_r^2+2a_r\cos (\xi)}-1\right)  . \nonumber
\end{equation}

\bibliographystyle{amsplain}

\begin{thebibliography}{000000000}

\bibitem[A83]{and:note} M.C.  Andr\'eief, Note sur une relation entre les int\'egrales d\'efinies des produits des fonctions,
M\'em. de la Soc. Sci. Bordeaux {\bf 2} (1883), 1--14.

\bibitem[BDS03]{bai-dei-stra:products} J. Baik, P. Deift, and E. Strahov, Products and ratios of characteristic polynomials of random Hermitian matrices, J. Math. Phys. {\bf 44} (2003), 3657--3670.

\bibitem[BSX95]{ber-sch-xu:multivariate} H. Berens, H.J. Schmid, and Y. Xu, Multivariate Gaussian cubature formulae, Arch. Math. {\bf 64} (1995), 26--32. 


\bibitem[BCM07]{ber-cac-mar:new} E. Berriochoa, A. Cachafeiro, and F. Marcell\'an, 
A new numerical quadrature formula on the unit circle,
Numer. Algorithms {\bf 44} (2007), 391--401. 

\bibitem[BCGM08]{ber-cac-gar-mar:new} E. Berriochoa, A.  Cachafeiro, J.M. Garc\'{\i}a-Amor, and F. Marcell\'an,
New quadrature rules for Bernstein measures on the interval $[ -1,1]$, 
Electron. Trans. Numer. Anal. {\bf 30} (2008), 278--290. 

\bibitem[CH15]{col-hub:moment} M. Collowald and E. Hubert, A moment matrix approach to computing symmetric cubatures,
2015. $\langle \text{hal-01188290v2} \rangle$

\bibitem[C97]{coo:constructing} R. Cools, Constructing cubature formulae: the science behind the art, Acta Numerica {\bf 6} (1997), 1--54.

\bibitem[CMS01]{coo-mys-sch:cubature} R. Cools, I.P. Mysovskikh, and H.J. Schmid, Cubature formulae and orthogonal polynomials,  J. Comput. Appl. Math. {\bf 127} (2001), 121--152.

\bibitem[BCDG09]{bul-cru-dec-gon:rational}
A. Bultheel, R. Cruz-Barroso, K. Deckers, and P. Gonz\'alez-Vera,
Rational Szeg\"o quadratures associated with Chebyshev weight functions, Math. Comp. {\bf 78} (2009),  1031--1059. 

\bibitem[BGHN01]{bul-gon-hen-nja:quadrature}
A. Bultheel, P. Gonz\'alez-Vera, E. Hendriksen, O. Nj\aa stad, Quadrature and orthogonal rational functions,  J. Comput. Appl. Math. {\bf 127} (2001), 67--91. 

\bibitem[DGJ06]{dar-gon-jim:quadrature}   L. Daruis, P. Gonz\'alez-Vera, M. Jim\'enez Paiz, Quadrature formulas associated with rational modifications of the Chebyshev weight functions, Comput. Math. Appl. {\bf 51} (2006), 419--430.

\bibitem[DR84]{dav-rab:methods} P.J. Davis and P. Rabinowitz, {\em Methods of Numerical Integration},
Second edition, Computer Science and Applied Mathematics, Academic Press, Inc., Orlando, FL, 1984. 

\bibitem[D06]{die:asymptotics} J.F. van Diejen,
Asymptotics of multivariate orthogonal polynomials with hyperoctahedral symmetry, in: {\em Jack, Hall-Littlewood and Macdonald Polynomials}, V.B. Kuznetsov and S. Sahi (eds.),
Contemp. Math. {\bf 417}, Amer. Math. Soc., Providence, RI, 2006, 157--169.

\bibitem[DMR07]{die-maz-ryo:berstein-szego} J.F. van Diejen, A.C. de la Maza, and S.  Ryom-Hansen, Bernstein-Szeg\"o polynomials associated with root systems, Bull. Lond. Math. Soc. {\bf 39} (2007), 837--847. 

\bibitem[DX14]{dun-xu:orthogonal} C.F. Dunkl and Y. Xu, {\em Orthogonal Polynomials of Several Variables}, Second edition,
Encyclopedia of Mathematics and its Applications,  vol. 155, Cambridge University Press, Cambridge, 2014.

\bibitem[F10]{for:log-gases} P.J. Forrester, {\em Log-Gases and Random Matrices}, London Mathematical Society Monographs Series, vol. 34, Princeton University Press, Princeton, NJ, 2010.

\bibitem[G81]{gau:survey} W.  Gautschi,
A survey of Gauss-Christoffel quadrature formulae, in: {\em  E.B. Christoffel: The Influence of his Work in Mathematics and Physical Sciences}, P.L. Butzer and F. Feh\'er (eds.), 
 Birkh\"auser, Basel, 1981, pp. 72--147. 

\bibitem[G93]{gau:gauss-type} W.  Gautschi,
Gauss-type quadrature rules for rational functions, {\em Numerical integration, IV (Oberwolfach, 1992)}, H. Brass and G. H\"ammerlin (eds.),  Internat. Ser. Numer. Math., vol. 112, Birkh\"auser, Basel, 1993, 111--130.

\bibitem[G01]{gau:use} W. Gautschi, The use of rational functions in numerical quadrature, J. Comput. Appl. Math. {\bf 133} (2001), 111--126.

\bibitem[G04]{gau:orthogonal} W.  Gautschi, {\em Orthogonal Polynomials: Computation and Approximation}, Numerical Mathematics and Scientific Computation, Oxford Science Publications. Oxford University Press, New York, 2004.

\bibitem[HM14]{hri-mot:discrete}  J.  Hrivn\'ak and L. Motlochov\'a,
Discrete transforms and orthogonal polynomials of (anti)symmetric multivariate cosine functions, SIAM J. Numer. Anal. {bf 52} (2014), 3021--3055. 

\bibitem[HMP16]{hri-mot-pat:cubature}  J.  Hrivn\'ak, L. Motlochov\'a, and J.  Patera,
Cubature formulas of multivariate polynomials arising from symmetric orbit functions, Symmetry {\bf 8} (2016), no. 7, Art. 63. 

\bibitem[H78]{hum:introduction}  J.E. Humphreys, {\em Introduction to Lie Algebras and Representation Theory}, Second printing, revised, Graduate Texts in Mathematics, vol. 9,
Springer-Verlag, New York-Berlin, 1978. 

\bibitem[IN06]{ise-nor:quadrature} A. Iserles and S.P. N\o rsett, Quadrature methods for multivariate highly oscillatory integrals using derivatives,
 Math. Comp. {\bf 75} (2006), 1233--1258. 

\bibitem[L91a]{las:jacobi} M. Lassalle, Polyn\^omes de Jacobi g\'en\'eralis\'es, C. R. Acad. Sci. Paris Sér. I Math. {\bf 312} (1991), 425--428.

\bibitem[L91b]{las:laguerre} \bysame, Polyn\^omes de Laguerre g\'en\'eralis\'es, C. R. Acad. Sci. Paris Sér. I Math. {\bf 312} (1991), 725--728.

\bibitem[L91c]{las:hermite} \bysame, Polyn\^omes de Hermite g\'en\'eralis\'es, C. R. Acad. Sci. Paris Sér. I Math. {\bf 313} (1991), 579--582.


\bibitem[LX10]{li-xu:discrete} H. Li and Y. Xu, 
Discrete Fourier analysis on fundamental domain and simplex of $A_d$ lattice in $d$-variables,
J. Fourier Anal. Appl. {\bf 16} (2010), 383--433. 

\bibitem[M92]{mac:schur}  I.G. Macdonald, Schur functions: theme and variations, S\'em. Lothar. Combin. {\bf 28} (1992), Art.
B28a.

\bibitem[M95]{mac:symmetric}  I.G. Macdonald, {\em Symmetric Functions and
Hall Polynomials}, Second Edition, Clarendon Press, Oxford, 1995.

\bibitem[M04]{meh:random}  M.L. Mehta, {\em Random Matrices}, Third Edition,  Elsevier/Academic Press, Amsterdam, 2004.

\bibitem[MP11]{moo-pat:cubature} R. Moody and J. Patera, 
Cubature formulae for orthogonal polynomials in terms of elements of finite order of compact simple Lie groups,
Adv. in Appl. Math. {\bf 47} (2011), 509--535. 


\bibitem[MMP14]{moo-mot-pat:gaussian}  R.V. Moody, L. Motlochov\'a, and J.  Patera,  Gaussian cubature arising from hybrid characters of simple Lie groups, J. Fourier Anal. Appl. {\bf 20} (2014), 1257--1290. 

\bibitem[NNSY00]{nak-nou-shi-yam:tableau} J. Nakagawa, M. Noumi, M. Shirakawa, and Y. Yamada, Tableau representation for Macdonald's ninth variation of Schur functions, in: {\em Physics and Combinatorics 2000}, A.N. Kirillov and N. Liskova (eds.), World Sci. Publ., River Edge, NJ, 2001, 180--195.

\bibitem[N90]{not:error} S.E.  Notaris, The error norm of Gaussian quadrature formulae for weight functions of Bernstein-Szeg\"o type, Numer. Math. {\bf 57} (1990),  271--283.


\bibitem[N00]{not:interpolatory} S.E. Notaris, 
Interpolatory quadrature formulae with Bernstein-Szeg\"o abscissae, in: {\em Integral and Integrodifferential Equations: Theory, Methods and Applications}, R.P. Agarwal and D. O'Regan (eds.), Series in Mathematical Analysis and Applications, Vol. 2, Gordon and Breach Science Publishers, Amsterdam, 1999, pp. 247--257.

\bibitem[NS12]{noz-saw:note} H. Nozaki and M. Sawa, Note on cubature formulae and designs obtained from group orbits, Canad. J. Math. {\bf 64} (2012), 1359--1377.

\bibitem[OLBC10]{olv-loz-boi-cla:nist} F.W.J. Olver, D.W. Lozier, R.F. Boisvert and C.W. Clark. (eds.),
{\em NIST Handbook of Mathematical Functions},
Cambridge University Press, Cambridge, 2010.

\bibitem[P93]{peh:remainder} F. Peherstorfer, On the remainder of Gaussian quadrature formulas for Bernstein-Szeg\"o weight functions, Math. Comp. {\bf 60} (1993), 317--325. 

\bibitem[P07]{pro:lie} C. Procesi, {\em Lie Groups. An Approach through Invariants and Representations}, Universitext, Springer, New York, 2007.

\bibitem[SX94]{sch-xu:bivariate} H.J. Schmid and Y. Xu,
On bivariate Gaussian cubature formulae,
Proc. Amer. Math. Soc. {\bf 122} (1994), 833--841. 

\bibitem[SV14]{ser-ves:jacobi} A.N. Sergeev and A.P. Veselov, Jacobi-Trudy formula for generalized Schur polynomials, Moscow Math. J. {\bf 14} (2014), 161--168.

\bibitem[S96]{sim:representations} B. Simon,
{\em Representations of Finite and Compact Groups}, 
Graduate Studies in Mathematics, vol. 10, American Mathematical Society, Providence, RI, 1996.

\bibitem[S92]{sob:cubature} S.L. Sobolev, 
{\em Cubature Formulas and Modern Analysis.
An Introduction}, Gordon and Breach Science Publishers, Montreux, 1992. 

\bibitem[SV97]{sob-vas:theory} S.L. Sobolev and V.L. Vaskevich,
{\em The Theory of Cubature Formulas}, Mathematics and its Applications, vol. 415, Kluwer Academic Publishers Group, Dordrecht, 1997.

\bibitem[S71]{str:approximate} A.H. Stroud,
{\em Approximate Calculation of Multiple Integrals}, 
Prentice-Hall Series in Automatic Computation, Prentice-Hall, Inc., Englewood Cliffs, N.J., 1971.

\bibitem[S75]{sze:orthogonal} G. Szeg\"o, {\em Orthogonal Polynomials}, Fourth Edition, American Mathematical Society, Colloquium Publications, vol. XXIII, American Mathematical Society, Providence, R.I., 1975.

\bibitem[X12]{xu:minimal-1} Y. Xu, Minimal cubature rules and polynomial interpolation in two variables, J. Approx. Theory {\bf 164} (2012), 6--30.

\bibitem[X17]{xu:minimal-2} \bysame, Minimal cubature rules and polynomial interpolation in two variables II, J. Approx. Theory {\bf 214} (2017), 49--68.

\bibitem[VV93]{van-van:quadrature} W. Van Assche and I. Vanherwegen,
Quadrature formulas based on rational interpolation,
Math. Comp. {\bf 61} (1993), 765--783. 

\end{thebibliography}

\end{document}